\documentclass[11pt,a4paper]{amsart}

\usepackage[top=2cm, bottom=2cm, left=2.5cm, right=2.5cm,twoside=false]{geometry} 
\makeatletter \renewenvironment{proof}[1][\proofname]{
	\par\pushQED{\qed}\normalfont
	\topsep6\p@\@plus6\p@\relax
	\trivlist\item[\hskip\labelsep\bfseries#1\@addpunct{.}]
	\ignorespaces}{
	\popQED\endtrivlist\@endpefalse} \makeatother

\usepackage{stmaryrd} 
\usepackage{amsfonts}
\usepackage{amsmath}
\usepackage{amssymb}
\usepackage{amsthm}
\usepackage{mathrsfs}
\usepackage{mathtools}
\usepackage[usenames,x11names]{xcolor}

\usepackage{soul}
\usepackage[normalem]{ulem}
\usepackage{csquotes}

\usepackage{graphicx}
\usepackage[all]{xy}

\usepackage{tikz}
\usepackage{tikz-cd}

\usepackage{enumerate}
\usepackage[shortlabels]{enumitem}
\setlist[enumerate,1]{label=(\alph*)}
\setlist[enumerate,2]{label=(\roman*), ref=(\alph{enumi}.\roman*)}
\newlist{parlist}{enumerate}{1}
\setlist[parlist]{leftmargin=0cm, itemindent=2\parindent, label=(\alph*), itemsep=0.5em}
\newcommand\litem[1]{\item{\emph{(#1)}\enspace}} 

\makeatletter
\@namedef{subjclassname@2020}{\textup{2020} Mathematics Subject Classification}
\makeatother

\usepackage
{hyperref}

\usepackage{cleveref}

\hypersetup{
	linkcolor=DodgerBlue3, 
	citecolor  = teal,
	urlcolor   = DodgerBlue4,
	colorlinks = true, 
	hyperfootnotes =false,
	bookmarksdepth=3,
	linktoc=all
}
\newcommand{\bP}{\mathbb{P}}
\newcommand{\bZ}{\mathbb{Z}}

\newcommand{\bA}{\mathbb{A}}
\newcommand{\bC}{\mathbb{C}}

\newcommand{\cA}{{\mathcal{A}}}

\newcommand{\cD}{{\mathcal{D}}}

\newcommand{\cE}{{\mathcal{E}}}
\newcommand{\cF}{{\mathcal{F}}}

\newcommand{\cL}{{\mathcal{L}}}

\newcommand{\cN}{{\mathcal{N}}}
\newcommand{\cO}{{\mathcal{O}}}

\newcommand{\cS}{{\mathcal{S}}}

\newcommand{\fs}{\mathfrak{s}}

\usepackage[scr=boondoxo]{mathalfa}
\newcommand{\kk}{\mathscr{k}}

\newcommand{\redd}{_{\textnormal{red}}} 

\newcommand{\gal}{\mathop{\textrm{Gal}}\nolimits}
\newcommand{\aff}{\mathop{\textrm{aff}}\nolimits} 
\newcommand{\sat}{\mathop{\textrm{sat}}\nolimits} 
\newcommand{\schsat}{\mathop{\textrm{sch-sat}}\nolimits} 
\newcommand{\sch}{\mathop{\textrm{sch}}\nolimits} 
\newcommand{\Hom}{\mathop{\textrm{Hom}}\nolimits}
\newcommand{\Pic}{\operatorname{Pic}}
\newcommand{\Ext}{\mathop{\operatorname{Ext}}\nolimits}
\newcommand{\Coh}{\mathop{\textrm{Coh}}\nolimits}
\newcommand{\Refll}{\mathop{\textrm{Ref}}\nolimits}
\newcommand{\Id}{\operatorname{Id}}
\newcommand{\Spec}{\mathop{\textrm{Spec}}\nolimits}

\newcommand{\Alb}{\operatorname{Alb}}

\newcommand{\Dalg}{\mathop{\mathcal{D}_{\textrm{alg}}}\nolimits}

\newcommand{\wt}[1]{\widetilde{#1}}
\newcommand{\wh}[1]{\widehat{#1}}
\newcommand{\ol}[1]{\overline{#1}}

\newcommand{\git}{/\!\!/}

\newcommand{\citestacks}[1]{\cite[\href{https://stacks.math.columbia.edu/tag/#1}{Tag #1}]{stacks-project}} 

\renewcommand{\leq}{\leqslant}
\renewcommand{\geq}{\geqslant}

\renewcommand{\to}{\longrightarrow}
\newcommand{\map}{\dashrightarrow}

\newcommand{\into}{\lhook\joinrel\longrightarrow}

\newcommand{\de}{\coloneqq} 

\newcommand{\cha}{\operatorname{char}}

\usepackage{aliascnt}

\newtheorem{LEM}{Lemma}[section]

\newtheorem*{THM*}{Theorem}
\newaliascnt{THM}{LEM}
\newtheorem{THM}[THM]{Theorem}
\aliascntresetthe{THM}

\newaliascnt{PROP}{LEM}
\newtheorem{PROP}[PROP]{Proposition}
\aliascntresetthe{PROP}

\newaliascnt{COR}{LEM}
\newtheorem{COR}[COR]{Corollary}
\aliascntresetthe{COR}

\newaliascnt{CON}{LEM}

\aliascntresetthe{CON}

\newtheorem{THMA}{Theorem}

\theoremstyle{definition}
\newaliascnt{EXM}{LEM}
\newtheorem{EXM}[EXM]{Example}
\aliascntresetthe{EXM}

\newaliascnt{REM}{LEM}
\newtheorem{REM}[REM]{Remark}
\aliascntresetthe{REM}

\newaliascnt{DEF}{LEM}
\newtheorem{DEF}[DEF]{Definition}
\aliascntresetthe{DEF}

\begin{document}

\title{Saturation of algebraic surfaces}

\author{Agnieszka Bodzenta}
\author{Tomasz Pe{\l}ka}
\author{Dario Wei{\ss}mann}

\address{Agnieszka Bodzenta: \newline
	Institute of Mathematics, 
	University of Warsaw, Banacha 2, 02-097 Warsaw, Poland.} \email{A.Bodzenta@mimuw.edu.pl} 
\address{Tomasz Pe{\l}ka: \newline
	Institute of Mathematics, 
	Jagiellonian University, Łojasiewicza 6, 30-348 Kraków, Poland.} \email{tomasz.pelka@uj.edu.pl}
\address{Dario Wei{\ss}mann: \newline
	Institute of Mathematics, 
	Polish Academy of Sciences, Śniadeckich 8, 00-656 Warsaw, Poland.} \email{dweissmann@impan.pl}

\subjclass[2020]{Primary 14C20; Secondary 14A20, 14F08, 18G80}


\begin{abstract}
	The saturation of an algebraic surface is
	the maximal open embedding with complement of dimension zero.
    	For schemes, it was introduced by the first named author
    	and A.~Bondal,  \cite{BodBon5},
    	who asked whether every saturated surface $X$
	is proper over its affinisation $\Spec H^{0}(X,\cO_{X})$.
    	We prove that this property holds whenever the affinisation is non-trivial.
    	If a saturated surface $X$ has trivial affinisation,
	we prove that the boundary of any compactification of $X$
	has at most two connected components, and this bound is optimal.
	Furthermore, we extend the results of \cite{BodBon5}
	from schemes to algebraic spaces; in particular, we prove that the saturation
	of a surface $X$ can be recovered from the category of reflexive sheaves on $X$.
\end{abstract}
\maketitle
\setcounter{tocdepth}{1}
\tableofcontents

\section{Introduction}

In this article we address the question of characterising saturated surfaces,
i.e. those to which one cannot \enquote{add more closed points}.
More precisely, we say that a surface $X$ is \emph{saturated}
if any open embedding $X\into  Y$ whose image has complement
of dimension zero is an isomorphism.

In the category of schemes, saturated surfaces were first introduced
in the article \cite{BodBon5} by the first named author and A.\ Bondal.
It was proved in \cite[Theorem 2.12]{BodBon5} that a surface $X$
is saturated if the affinisation morphism
$X\to X^{\aff}\de \Spec H^0(X, \cO_X)$ is proper (see Section~\ref{ssec:def_of_aff}).
The converse is false e.g.\ if $X$ is the schematic locus
of a proper algebraic space which is not a scheme, see \cite[Example 2.13]{BodBon5}.
This article is motivated by a natural question raised by A.\ Bondal:
 does the converse hold when, instead of schemes, one considers algebraic spaces?
In this article, we answer this question (Theorem \ref{thm:intro_sat_implies_proper})
and develop the theory of saturation for algebraic spaces in parallel
to the one developed for schemes in \cite[\S 2]{BodBon5}.
\smallskip

We work over an algebraically closed field $\kk$.
By a \emph{surface} we mean an irreducible, separated, normal \emph{algebraic space}
of finite type and dimension two over $\Spec\kk$.
An open embedding of surfaces is \emph{big}
if its image has complement of dimension zero.
We say that a surface $X$ is \emph{saturated} (Definition \ref{def:saturation})
if every big open embedding $X\into Y$ is an isomorphism.
To avoid the risk of confusion with the terminology of \cite{BodBon5}
concerning schematic surfaces (i.e., surfaces which are schemes),
we say that a schematic surface $X$ is \emph{scheme-saturated}
if every big open embedding of schematic surfaces $X\into Y$ is an isomorphism.

We give the following answer to the question of A.\ Bondal.

\begin{THMA}[Theorem \ref{thm:proper-surjective}]
\label{thm:intro_sat_implies_proper}
	Let $X$ be a surface such that $\dim X^{\aff}>0$.
	Then $X$ is saturated if and only if the affinisation morphism
	$X\to X^{\aff}$ is proper.
\end{THMA}

We emphasise that Theorem \ref{thm:intro_sat_implies_proper}
requires passing to the category of algebraic spaces.
Indeed, in Section \ref{ssec:examples-scheme-sat} we give examples of smooth
surfaces $X$ with $\dim X^{\aff}=0,1,2$ which are scheme-saturated,
but are non-proper over their affinisations.
Nonetheless, a weaker property that $X\to X^{\aff}$
is surjective holds for scheme-saturated surfaces, too (Proposition \ref{prop:surjective-2}).

Theorem \ref{thm:intro_sat_implies_proper} motivates further study
of surfaces for which the answer to A.\ Bondal's question is negative,
i.e.\ non-proper, saturated surfaces $X$ with $\dim X^{\aff}=0$.
Examples \ref{ex:ruled}, \ref{ex:Serre}, and \ref{ex:Hironaka}\ref{item:Hironaka-0}
show that such surfaces do exist
(in fact, they can be constructed in any class of birational equivalence,
see Proposition \ref{prop:SA-minimal}\ref{item:SA-birational-converse}),
at least if $\kk$ is not an algebraic closure of a finite field.
We describe them in terms of the boundaries of their compactifications, as follows.

Let $X$ be a surface.
Let $\ol{X}$ be a proper surface containing $X$
as an open subspace and consider $D\de \ol{X}\setminus X$.
Clearly, if $X$ is saturated, then $D$ cannot contain isolated points
nor connected components which can be contracted to points by a proper birational morphism.
In Proposition~\ref{prop:saturation-criterion}, we prove that the converse holds, too.
Hence by Artin's criterion (see Lemma~\ref{lem:intersection-product}\ref{item:D-Artin}),
$X$ is saturated if and only if $D$
is a Weil divisor whose connected components are not negative definite.

Assume $X$ is saturated.
If $D$ is not negative semi-definite,
then $\dim X^{\aff}=2$ by Proposition~\ref{prop:numerics}\ref{item:numerics-2}:
in fact, $D$ contains the support of a big and nef divisor,
cf.\ \cite[Proposition 3.2]{Schroer}. Assume $\dim X^{\aff}<2$.
Then every connected component $D_0$ of $D$ is negative semi-definite,
but not negative definite. We say that such $D_0$ is \emph{of fibre type},
because it has numerical properties of a fibre of some fibration of $\ol{X}$,
see Lemma \ref{lem:fibre-type}.
It is easy to see that $D_0$ supports a fibre
if and only if $\dim X^{\aff}=1$,
in which case the fibration restricts to $X\to X^{\aff}$,
see Proposition \ref{prop:numerics}\ref{item:numerics-1}.

Therefore, we have $\dim X^{\aff}=0$
(i.e.\ the answer to A.\ Bondal's question is negative)
if and only if the connected components of $D$ are \emph{false fibres}, i.e.,
they are of fibre type, but do not support fibres.
Examples \ref{ex:Serre} and \ref{ex:Hironaka} show that the cases
$\dim X^{\aff}=1,0$ cannot be distinguished by numerical properties of $D$.
Nonetheless, we point out a sharp difference between them.

\begin{THMA}[Corollary \ref{cor:two-false-fibres}]\label{thm:intro-two-false-fibres}
	Let $X$ be a saturated surface whose affinisation $X^{\aff}$ is a point.
	Let $\ol{X}$ be a proper surface containing $X$ as an open subspace,
	and let $D\de \ol{X}\setminus X$.
	Then $D$ has at most two connected components.
\end{THMA}

Of course, if $\dim X^{\aff}=1$,
then $D$ can have arbitrarily many connected components
(simply remove from $X$ more fibres of $X\to X^{\aff}$,
they are proper by Theorem \ref{thm:intro_sat_implies_proper}).

Example \ref{ex:ruled}\ref{item:ex-ruled-2} shows that the number $2$
in Theorem \ref{thm:intro-two-false-fibres} is optimal:
there do exist saturated surfaces with trivial affinisation
and two connected components of the boundary.
However, in Proposition \ref{prop:two-false-fibres} we show that
this case is strongly restricted:
for instance, if $D$ has two smooth connected components $D_0$, $D_1$,
then by Proposition \ref{prop:two-false-fibres}\ref{item:false-fibre-genus-irregularity}
the irregularity $q(\ol{X})$ cannot vanish,
the genus of $D_i$ is at least $q(\ol{X})$,
and its normal bundle is non-torsion (Remark \ref{rem:normal-bundle}).
On the other hand, if $D$ is connected,
then it can happen that $q(\ol{X})=0$ (Example \ref{ex:Hironaka})
or $D$ has trivial normal bundle (Example \ref{ex:Serre}).
As a side result, in Section \ref{ssec:SA-min} we reformulate
Theorem \ref{thm:intro-two-false-fibres} as a geometric characterization
of non-proper saturated surfaces with $\dim X^{\aff}=0$ (or $1$).

The question whether a divisor $D$ of fibre type
is a fibre of a fibration $\ol{X}\to B$ was raised by Mumford
in \cite[p.\ 339]{Mumford-Enriques-1},
and studied by Sakai \cite{Sakai-D-dimension} for $D$ of genus $1$,
see Remark \ref{rem:Sakai}.
By Theorem \ref{thm:intro-two-false-fibres},
$D$ is a fibre if and only if $\ol{X}\setminus D$
contains two disjoint divisors of fibre type,
or in other words, a \enquote{candidate} for a pullback of an ample divisor on $B$.
From this point of view, Theorem \ref{thm:intro-two-false-fibres}
is similar to \cite[Theorem 3.3]{Schroer},
which asserts that a negative definite divisor $D$
is a fibre of a proper birational morphism $\ol{X}\to \ol{Y}$
of schematic surfaces if and only if $\ol{X}\setminus D$
contains an effective \enquote{candidate} for a pullback
of an ample divisor on $\ol{Y}$, see Remark \ref{rem:Schroer}.

\smallskip

Our second goal is to prove that saturation of algebraic spaces
behaves at least as well as scheme-saturation developed in \cite{BodBon5}.
To state our results, we introduce further terminology.

The \emph{saturation} of a surface $X$ is a big open embedding
$\eta_{X}\colon X\into X^{\sat}$ into a saturated surface
(Definition \ref{def:saturation}):
intuitively speaking, it adds all the closed points that can be added to $X$.
Similarly, \emph{scheme-saturation} of a schematic surface $X$
is a big open embedding into a scheme-saturated surface $X\into X^{\schsat}$.
These two notions are related: for any surface $X$
the scheme-saturation of the schematic locus of $X$
is the schematic locus of the saturation $X^{\sat}$
(Proposition \ref{prop:sat_and_sch_commute}).
In particular, the scheme-saturation $X^{\schsat}$ of a schematic surface $X$
is a big open subspace of $X^{\sat}$,
but they may not be equal (as shown by examples in Section \ref{ssec:examples-scheme-sat}).

In Section \ref{sec:saturation}, we check that
basic categorical properties of scheme-saturation
proved in \cite[\S\S 2.3, 2.4]{BodBon5} carry over to saturation.
For every surface $X$ its saturation $\eta_{X}\colon X\into X^{\sat}$
exists (Corollary \ref{cor:saturation-exists})
and satisfies the following universal property
(Proposition \ref{prop:univ-property}):
for every big open embedding $i\colon X\into Y$
there is a unique big open embedding $j\colon Y\into X^{\sat}$
such that $j\circ i=\eta_X$.
In particular, the saturation is unique up to a unique isomorphism
(Corollary \ref{cor:sat_unique}).
In fact, $\eta_{X}$ is the adjunction unit $X \to \mathfrak{s}_* \mathfrak{s}X$,
where $\mathfrak{s}$ 
is the canonical functor from the category 
of surfaces and open embeddings to its localisation in all morphisms,
and $\mathfrak{s}_{*}$ is its left adjoint (Theorem \ref{thm:adjoint}).
The proof of these results follows the steps of \cite{BodBon5}
except for the construction of saturation, which cannot be adapted,
see Section \ref{ssec:sat-construction} for details.
 
Next, we prove that if $X\to Y$ is a proper surjective morphism,
then $X$ is saturated if and only if $Y$ is;
in fact saturation is functorial with respect to proper surjective morphisms
(Theorem \ref{thm:functoriality}).
As shown by Example \ref{ex:covers},
this property fails for scheme-saturation.

Finally, we prove the analogue of the main theorem of \cite{BodBon5}:

\begin{THMA}\label{thm:intro_sat_from_Ref}
	Let $X$ be a surface.
	Then the saturation $\eta_{X}\colon X\into X^{\sat}$
	can be recovered from the category of reflexive sheaves on $X$,
	as described in Theorem \ref{thm:sat_from_Ref}.
\end{THMA} 

In fact, Theorem \ref{thm:intro_sat_from_Ref} improves \cite{BodBon5},
as it recovers not only the surface $X^{\sat}$, but also the morphism $\eta_{X}$.
Its proof relies on the results of
\cite{BodBon5} combined with the methods of \cite{CalPir}.

\subsection*{Structure of the paper.}

In the preliminary Section \ref{sec:gluing}
we recall some general properties of algebraic spaces.
In Section \ref{ssec:pushout} we show how to glue surfaces
along big open subspaces, thus proving that the category $\Dalg$
of surfaces and big open embeddings admits pushouts (Lemma~\ref{lem:push-outs_in_Dalg}).
In Section \ref{ssec:contractions} we recall 
a standard fact about gluing birational contractions.

In Section \ref{sec:saturation} we formally introduce saturation of surfaces
(Definition \ref{def:saturation}) and prove its basic properties.
In Section \ref{ssec:sat-construction} we prove that the saturation exists
and can be constructed from a compactification as described above,
see Proposition \ref{prop:saturation-criterion}.
In Section \ref{ssec:saturation_univ-property}
we prove its universal property in the category $\Dalg$.

In Section \ref{ssec:adjunction},
we prove that all morphisms in $\Dalg$ can be inverted,
i.e., we can consider the localisation $\wh{\Dalg} = \Dalg[\Hom^{-1}]$
together with the canonical functor $\fs\colon \Dalg \to \wh{\Dalg}$.
We show in Theorem \ref{thm:adjoint} that $\fs$ admits the left adjoint $\fs_*$
and that the saturation is the $\fs_* \dashv \fs$ adjunction unit 
(hence the notation $\eta_X \colon X \into X^{\sat}$). 

In Section \ref{ssec:schematic-locus} we relate scheme-saturation with saturation
via $(X^{\sat})^{\sch}=(X^{\sch})^{\schsat}$ (Proposition \ref{prop:sat_and_sch_commute}). 
Next, in Section \ref{ssec:functoriality}
we prove that saturation is functorial 
under proper surjective morphisms. 
Using the Stein factorization, we reduce this statement to the cases 
of finite and proper birational morphisms, settled in 
the Propositions \ref{prop:covers} and 
\ref{prop:functoriality-birational}, respectively. 
We note that the latter holds for the scheme-saturation, too, 
but for the former it is necessary to consider algebraic spaces (Example \ref{ex:covers}).

Finally, in Section \ref{ssec:sat_from_ref} 
we describe how the saturation $\eta_X\colon X\into X^{\sat}$
is determined by $X$ and the category of reflexive sheaves on $X$,
see Theorem~\ref{thm:intro_sat_from_Ref}. 

Theorem \ref{thm:intro_sat_implies_proper} is proved in Section \ref{sec:affinisation}.
In Section \ref{ssec:def_of_aff} we recall basic properties of the affinisation.
In Section \ref{ssec:intersection form}
we reduce the problem to the smooth setting using Chow's lemma, 
resolution of surface singularities,
and functoriality of saturation with respect to proper birational morphisms. 
Theorem \ref{thm:intro_sat_implies_proper} is proved in Section \ref{ssec:prop_of_aff}.
Lastly, in Section~\ref{ssec:surj_of_aff}
we prove that for scheme-saturated surfaces the affinisation morphism is always surjective,
even if it is not proper (Proposition \ref{prop:surjective-2}).

In Section \ref{sec:sat_open_sub} we prove Theorem \ref{thm:intro-two-false-fibres}
and related results announced above.
In Section~\ref{ssec:num_boundary} we recall some properties 
of divisors of fibre type, and describe boundaries of
saturated surfaces (Proposition \ref{prop:numerics}). 
In Section \ref{ssec:false-fibres} we prove
Theorem \ref{thm:intro-two-false-fibres} and compare it with 
some results in the literature (Remarks \ref{rem:Mumford}, \ref{rem:Sakai}, \ref{rem:Schroer}). 
Finally, in Section~\ref{ssec:SA-min} we deduce 
some characterisations of saturated surfaces $X$ 
with $\dim X^{\aff}=1$ (Proposition~\ref{prop:dimXaff=1})
and $\dim X^{\aff}=0$ (Theorem~\ref{thm:dimXaff=0}). 

We conclude our discussion with a list of examples in Section \ref{sec:exampl}.
Examples in Section \ref{ssec:examples-dimXaff-0} show that the assumption 
$\dim X^{\aff}>0$ in Theorem \ref{thm:intro_sat_implies_proper}
is necessary for a saturated surface $X$
to be proper over its affinisation.
Example \ref{ex:Hironaka} in Section \ref{ssec:examples-numerics}
shows that numerical properties of the boundary divisor
cannot distinguish between $\dim X^{\aff} =1$ and $\dim X^{\aff} =0$.
In Section \ref{ssec:examples-scheme-sat}
we construct scheme-saturated surfaces $X$
which are not proper over $X^{\aff}$ 
for all possibilities $\dim X^{\aff}\in \{0,1,2\}$,
thus showing that Theorem \ref{thm:intro_sat_implies_proper}
does not hold for scheme-saturated surfaces.

\subsection*{Acknowledgement}
The first named author was  partially supported
by Polish National Centre (NCN) contract number 2021/41/B/ST1/03741. 

\subsection*{Notation} 
We work over an algebraically closed field $\kk$. 
A \emph{curve} is a separated integral scheme 
of dimension one and of finite type over $\kk$.
A \emph{surface} is an irreducible, normal, 
separated algebraic space of dimension two and of finite type over $\kk$. 
A \emph{schematic surface} is a surface which is also a scheme.
By a \emph{divisor} on a surface we always mean a Weil divisor,
and we often identify reduced divisors with their supports. 

An open embedding of surfaces is \emph{big}
if its image has codimension at least two. 
We denote by $\Dalg$ (resp.\ $\cD$) the category of surfaces 
(resp.\ schematic surfaces) and big open embeddings.
For a surface $X$, we denote by $X^{\sch}$ its schematic locus,
and by $\eta_{X}\colon X\to X^{\sat}$ and $\alpha_{X}\colon X\to X^{\aff}$
the saturation and affinisation morphisms,
see Definition \ref{def:saturation}
and Section \ref{ssec:def_of_aff}, respectively.

\section{Gluing surfaces}\label{sec:gluing}

To prove the existence of saturation,
we need to \enquote{glue} to $X$ the missing points of its saturated model.
With this in mind, we gather some useful lemmas concerning gluing of algebraic spaces.
The main result of this section is Lemma \ref{lem:push-outs_in_Dalg},
which allows to glue surfaces along big open subspaces. 
It is a consequence of the corresponding result for schemes 
\cite[Proposition 2.3]{BodBon5} and known properties of algebraic spaces
and their covers, which we recall now.

We say that an open subspace $U\subset X$ of an algebraic space is \emph{big}
if the complement $X\setminus U$ has codimension at least two in $X$.
We say that $j \colon U \into X$ is a \emph{big open embedding}
if it is an open embedding with big image.
For an algebraic space $X$ we denote by $X^{\sch}$ the \emph{schematic locus} of $X$, 
i.e., the set of geometric points in $X$ that admit an affine open neighbourhood,
see \citestacks{03JG} for details.
By \citestacks{0ADD}, $X^{\sch}\subseteq X$ is a big open subspace.
If $X$ is normal, then the smooth locus of $X$ 
is a big open subspace of $X$, too.

\subsection{Covers}\label{ssec:covers}

A morphism $Y\to X$ of normal separated irreducible 
algebraic spaces is a \emph{cover} if it is
finite, surjective, and generically \'etale. 
A cover $Y\to X$ is \emph{Galois} if the function field extension
$\kappa(Y)/\kappa(X)$ is Galois,
where $\kappa(X)$ is the field of rational functions, 
see \citestacks{0END}.
We refer to the group $\gal(\kappa(Y)/\kappa(X))$ 
as the \emph{Galois group} of the cover $Y\to X$. 

\begin{LEM}[Galois covers]
	\label{lem:lifting-Galois}
	Let $Y\to X$ be a Galois cover of surfaces with Galois group $G$.
	\begin{enumerate}
		\item\label{item:G-acts}
			The group $G$ acts on $Y$ over $X$
			and every automorphism of $Y$ over $X$ arises this way.
		\item\label{item:G-quotient} 
			The surface $X$ is the coarse quotient 
			$Y\git G$ of $Y$ by $G$,
			in the sense that it is universal among
			$G$-equivariant morphisms $Y\to Z$ to a space $Z$
			with the trivial $G$-action and geometric points of $X$
			correspond to $G$-orbits of geometric points of $Y$. 
	\end{enumerate}
\end{LEM} 

\begin{proof}
	\ref{item:G-acts} 
	Since $Y\to X$ is Galois,
	$Y$ is the integral closure of $X$ in $\kappa(Y)$,
	hence \ref{item:G-acts} follows from 
	the universal property of normalization \citestacks{0823}. 
	
	\ref{item:G-quotient}
	The result of Keel--Mori \cite{keel-mori}
	shows that there exists a coarse moduli space
	of $[Y/G]$, call it $\wt{X}$.
	Since $Y$ is a normal and separated algebraic space, so is $\wt{X}$.
	The action of $G$ on $Y$ is generically	free.
	We conclude that $\kappa(\wt{X})=\kappa(X)$,
	so the natural morphism $\wt{X}\to X$ is birational. 
	Since $Y\to X$ and $Y\to \tilde{X}$ are finite and surjective,
	so is $\wt{X}\to X$. 
	Thus, $\wt{X}\to X$ is a finite birational morphism
	of normal separated algebraic spaces, hence 
	an isomorphism by Zariski's Main Theorem \citestacks{082K}.  
\end{proof}

\begin{LEM}[Big open embeddings lift to covers]
	\label{lem:G-lifting}
	Let $Y\to X$ be a finite morphism of surfaces. 
	Let $X\into X'$ be a big open embedding.
	Let $Y'$ be the integral closure of $X'$ in $\kappa(Y)$,
	so we have a commutative diagram
	\begin{equation*}
		\begin{tikzcd}
			Y \ar[d] \ar[r] & Y' \ar[d]\\
			X \ar[r,hook] & X'.
		\end{tikzcd}
	\end{equation*}
	Then $Y'\to X'$ is finite, and $Y\to Y'$ is a big open embedding.
	
    	Moreover, if $Y\to X$ is a Galois cover,
	then $Y'\to X'$ is a Galois cover 
    	with the same Galois group $G$,
	and $Y\into Y'$ is $G$-equivariant.
\end{LEM}	
\begin{proof} 
	The natural morphism $Y'\to X'$ from the integral closure 
    	is finite by \citestacks{0BB0}.
    	By \citestacks{0ABP} we have $Y=X \times_{X'} Y'$,
    	so $Y\to Y'$ is an open embedding.
    	The complement $Y'\setminus Y$ is the preimage of the finite set
    	$X'\setminus X$ by a finite morphism,
    	hence $Y'\setminus Y$ is finite,
    	i.e., $Y\into Y'$ is a big open embedding. 
	
	Assume that $Y\to X$ is a Galois cover.
	We have $\kappa(Y')=\kappa(Y)$ and $\kappa(X')=\kappa(X)$,
    	so $Y'\to X'$ is a Galois cover with the same Galois group $G$.
   	By Lemma \ref{lem:lifting-Galois}\ref{item:G-quotient} we have that
	$X$ is the quotient $X=Y\git G$ 
	and $X'=Y'\git G$, so $Y\into Y'$ is $G$-equivariant.     
\end{proof}

\begin{LEM}[{Algebraic spaces as quotients of schemes, see \cite[Corollaire 16.6.2]{champs-algebrique}}]
	\label{lem:Galois}
    	Let $X$ be a surface.
    	Then there exists a schematic surface $Y$
    	and a Galois cover $Y\to X$ with Galois group $G$.
    	As a consequence, $X$ is the coarse moduli space
	of the quotient stack $[Y/G]$.
\end{LEM}

\begin{proof}
    	By \cite[\href{https://stacks.math.columbia.edu/tag/0GUN}{Tag 0GUN}]{stacks-project}  
	there is a finite surjective generically \'etale morphism $Y\to X$,
	where $Y$ is a scheme.
	Replacing $Y$ by the normalisation of an irreducible component 
	dominating $X$ equipped with the reduced subscheme structure,
	we may assume that $Y$ is normal and integral.
	Furthermore, the function field extension 
	$\kappa(Y)/\kappa(X)$ is separable
	since the morphism $Y\to X$ is generically \'etale.
	Replacing $Y$ by its integral
	closure in the Galois closure
	of 
	$\kappa(Y)/\kappa(X)$ we may 
	assume that $Y\to X$ is Galois.
	The remaining statement follows from 
	Lemma \ref{lem:lifting-Galois}\ref{item:G-quotient}.
\end{proof}

\subsection{Pushouts}\label{ssec:pushout} 
Let $\Dalg$ be the category of surfaces with big open embeddings, 
and $\cD$ its full subcategory of schematic surfaces.
We have the following extension of \cite[Proposition 2.3]{BodBon5}.

\begin{LEM}[Pushouts in $\Dalg$ and $\cD$]
	\label{lem:push-outs_in_Dalg}
	Consider big open embeddings 
	$X_0 \into X_1$, $X_0 \into X_2$ of surfaces.
	Then there exists a pushout $X = X_1 \sqcup_{X_0} X_2$
	in the category of separated algebraic spaces,
 	and the natural maps $X_1\to X$ and $X_2 \to X$ 
	are big open embeddings of surfaces. 
	Moreover, the surface $X$ is schematic if $X_0,X_1,X_2$ are schematic.

	In particular, the categories $\Dalg$ and $\cD$ admit pushouts.
\end{LEM}
\begin{proof}
	By Lemma \ref{lem:Galois}, the surfaces $X_i$ for $i=0,1,2$
	are quotients $X_i=Y_i/G_i$ of schematic surfaces $Y_i$ 
	by the Galois groups $G_i=\gal(\kappa(Y_i)/\kappa(X_i))$. 
	Let $L$ be a Galois extension of $\kappa(X_0)$ dominating
	all the Galois extensions $\kappa(Y_0),\kappa(Y_1),$ and $\kappa(Y_2)$,
	and let $G=\gal(L/\kappa(X_0))$ be its Galois group.
	If $X_i$ is already schematic we take $Y_i=X_i$ and $L=\kappa(X_i)$. 
	
	Replacing $Y_0,Y_1,Y_2$ by their integral closures in $L$,
	by Lemma \ref{lem:G-lifting} we obtain  
	$G$-equivariant big open embeddings $Y_{0}\into Y_{1},Y_2$
	of schematic surfaces which fit into commutative diagrams  
	\begin{equation*}\begin{tikzcd} 
		Y_0 \ar[r,hook] \ar[d] & Y_i \ar[d]\\
		X_0 \ar[r,hook] & X_i       
	\end{tikzcd}\end{equation*}
	for $i=1,2$. 
	By \cite[Proposition 2.3]{BodBon5},
	big open embeddings of schematic surfaces
	form a left multiplicative system,
	so there is a schematic surface $Y$ 
	and a commutative diagram
	\begin{equation*}
		\begin{tikzcd}
			Y _0 \ar[r,hook] \ar[d,hook] & Y_{1} \ar[d]\\
			Y_{2} \ar[r] & Y.
		\end{tikzcd}
	\end{equation*}
    	Moreover, loc.\ cit.\ constructs $Y$ so that it is 
    	the union of images of $Y_1$ and $Y_2$,
	whose intersection contains the image of $Y_0$.
    	In particular, $Y_{i}\to Y$ is a big open embedding for $i=0,1,2$.
    
    	We claim that $Y$ is the pushout $Y_1\sqcup_{Y_0} Y_2$
 	in the category of separated algebraic spaces.
    	Consider a pair of morphisms $f_i\colon Y_i \to Z$, $i=1,2$,
    	to a separated algebraic space $Z$ which agree on $Y_0$.
	Since $Z$ is separated and $Y_0$ is dense in $Y_1\cap Y_2$,
	the morphisms $f_1$ and $f_2$ agree on $Y_1 \cap Y_2$. 
	It follows that $f_1$ and $f_2$
	glue to a morphism $f\colon Y_{1}\cup Y_{2}=Y\to Z$,
	which, again, is unique because $Z$ is separated 
    	and $f$ is already defined on a dense open subspace $Y_0\subseteq Y_1\cap Y_2$
    	of the reduced scheme $Y$, see \citestacks{0EMR} or \cite[Corollary 9.9]{GorWed}.
   	This proves the claim. 
  
	The action of $G$ on $Y_i$ induces an action of $G$ on $Y$,
    	and the surface $X\de Y\git G$, which exists by \cite{keel-mori},
    	is the pushout $X_1 \sqcup_{X_0} X_2$
	in the category of separated algebraic spaces.
    	Furthermore, for $i=0,1,2$, the $G$-equivariant
	big open embedding $Y_i\into Y$
    	induces a big open embedding $X_i\into X$
	on the level of $G$-quotients, as needed. 
\end{proof}

\subsection{Contracting Weil divisors}\label{ssec:contractions}
We recall a standard definition, cf.\ \cite{Schroer}.

\begin{DEF}\label{def:contractible}
	Let $D$ be a proper Weil divisor on a (schematic) surface $X$.
	We say that $D$ is (\emph{schematically}) \emph{contractible}
	if there is a proper surjective morphism $X\to Y$ of (schematic)
	surfaces which is an isomorphism on $X\setminus D$
	and maps $D$ to a finite set. 
\end{DEF}
%
We remark that by Artin's criterion,
see Lemma \ref{lem:intersection-product}\ref{item:D-Artin},
$D$ is contractible if and only if it is negative definite.
Schematic contractibility is more subtle,
see \cite[Theorem 3.3]{Schroer} or Lemma \ref{lem:Hironaka}.
The following \enquote{gluing} argument is well known, cf.\ \cite[Lemma 3.1]{Schroer}.

\begin{LEM}\label{lem:patching}
	Let $D$ be a proper Weil divisor on a (schematic) surface $X$,
	and let $U$ be an open neighborhood of $D$ in $X$.
	Then $D$ is (schematically) contractible as a divisor on $X$
	if and only if it is (schematically) contractible as a divisor on $U$.
	In particular, $D$ is (schematically) contractible if and only if
	the same holds for every connected component of $D$.
\end{LEM}
\begin{proof}
	Assume that $D$ is (schematically) contractible as a divisor on $U$;
	and let $f\colon U\to V$ be the proper birational morphism
	as in Definition \ref{def:contractible}.
	Let $g\colon X\to Y\de (X\setminus D)\sqcup_{U\setminus D} V$
	be the proper morphism obtained by gluing $\Id_{X\setminus D}$
	with $f$ along $U\setminus D$.
	It is an isomorphism on $X\setminus D$ and maps $D$ to a finite set. 
	It remains to prove that $Y$ is separated.
	
	By the valuative criterion \citestacks{0ARI},
	we need to prove that for every discrete
	valuation ring $R$ with fraction field $K$,
	and every morphism $\eta_0\colon \Spec K\to Y$
	there is at most one lift $\eta\colon \Spec R\to Y$
	such that $\eta\circ\iota=\eta_0$,
	where $\iota\colon \Spec K\to  \Spec R$ is the canonical morphism.
	If $\eta_{0}$ maps $\Spec K$ to a closed point of $Y$,
	then $\eta$ must map $\Spec R$ to the same closed point,
	so we can assume that the image of $\eta_0$ is one-dimensional.
	Since $Y\setminus (X\setminus D)=f(D)$ has dimension zero,
	we get that $\eta_{0}$ factors through $X\setminus D$.
	Let $\eta_{1},\eta_{2}$ be two lifts of $\eta_0$.
	If they both factor through the same patch ($X\setminus D$ or $V$),
	then they are equal since both patches are separated.
	Hence we can assume that $\eta_{1}$ and $\eta_2$ factor through
	$X\setminus D$ and $V$, respectively.
	In particular, $\eta_{1}\colon \Spec R\to X\setminus D\into X$ is a lift of
	$\eta_{0}\colon \Spec K\to X\setminus D \into X$ to a morphism $\Spec R\to X$.
	Since $f\colon U\to V$ is proper,
	$\eta_{2}$ lifts to $\ol{\eta}_{2}\colon \Spec R\to U\into X$,
	which is another lift of $\eta_0\colon \Spec K\to X$ to $\Spec R\to X$.
	Since $X$ is separated, we conclude that $\eta_{1}=\ol{\eta}_{2}$.
	Hence $\ol{\eta}_{2}$ factors through $X\setminus D$,
	and the result follows since $X\setminus D$ is separated.
\end{proof}

\section{Saturation of surfaces}\label{sec:saturation}

\begin{DEF}[Saturation]\label{def:saturation}
	We say that a surface $X$ is \emph{saturated} if every
	big open embedding of surfaces $X\into Y$ is an isomorphism.
	A \emph{saturation} of $X$ is a big open embedding
	$X\into X^{\sat}$ into a saturated surface.
\end{DEF}

Similarly, we say that a schematic surface is
\emph{scheme-saturated} if every big open embedding $X\into Y$
of \emph{schematic} surfaces $X\into Y$ is an isomorphism;
and define \emph{scheme-saturation} $X\into X^{\schsat}$
as a big open embedding into a scheme-saturated surface.
This notion was introduced in \cite[\S 2]{BodBon5} as \emph{codim-2-saturation}.
Therefore, Definition \ref{def:saturation} is a direct analogue
of the one in \cite{BodBon5} to the category of algebraic spaces.

Clearly, a saturated schematic surface is scheme-saturated,
but not conversely, see the examples in Section \ref{ssec:examples-scheme-sat}.

\smallskip 

In this section, we prove that the saturation exists
(Corollary \ref{cor:saturation-exists}),
is unique (Corollary \ref{cor:sat_unique}),
and has categorical properties analogous
to those proved for scheme-saturation in \cite{BodBon5}
(Proposition \ref{prop:univ-property} and Theorem \ref{thm:adjoint}).
The latter are proved in a similar way as in \cite{BodBon5},
but the existence proof is different.
In Theorem \ref{thm:functoriality} we prove that saturation
behaves well with respect to proper surjective morphisms:
this has no counterpart in \cite{BodBon5},
as it fails for scheme-saturation (Example \ref{ex:covers}).
In Section \ref{ssec:sat_from_ref} we show how to reconstruct the saturation
$X\to X^{\sat}$ from the category of reflexive sheaves on $X$,
which improves the main result of \cite{BodBon5}.

\subsection{Construction of a saturation}
\label{ssec:sat-construction}
For a schematic surface $X$,
its scheme-saturation $X^{\schsat}$
was constructed in \cite[\S 2.3]{BodBon5}
as a colimit of affinisations of its quasi-affine open subschemes.
As such, this construction does not directly generalize to algebraic spaces.
However, \cite[Theorem 2.14]{BodBon5} shows that $X^{\schsat}$
can be constructed from a compactification $\ol{X}$ of $X$, as follows:
first contract all connected components of the boundary $\ol{X}\setminus X$
which are schematically contractible Weil divisors,
and then define $X^{\schsat}$ as $\ol{X}$
minus the divisorial part of the boundary.

In this section, we show that the latter approach
does generalize to algebraic spaces.
This way, in Proposition \ref{prop:saturation-criterion}
we give an independent proof of both \cite[Theorem 2.14]{BodBon5}
and its direct analogue for algebraic spaces,
and deduce the existence of saturation in Corollary \ref{cor:saturation-exists}.

\begin{PROP}
\label{prop:saturation-criterion}
	Let $X$ be an open subspace of a proper surface $\ol{X}$.
	Let $D=\ol{X}\setminus X$.
	Then $X$ is saturated if and only if $D$ is a Weil divisor
	with no contractible connected component.
	
	Moreover, if $\ol{X}$ is a scheme,
	then $X$ is scheme-saturated if and only if
	$D$ is a Weil divisor with 
	no schematically contractible connected component.
\end{PROP}
\begin{proof}
	We can assume that $D$ has no (schematically)
	contractible connected components:
	indeed, if $D$ has one, then we contract it,
	replace $\ol{X}$ by its image, and repeat if needed.
	After this reduction, we need to prove that
	$X$ is (scheme-)saturated if and only if $D$ has pure dimension~$1$.
	
	Assume that a point $x\in \ol{X}$ is a connected component of $D$.
	The big open embedding $X\into \ol{X}\setminus (D\setminus \{x\})$
	is not an isomorphism, hence $X$ is not saturated, as needed.
	
	Conversely, assume $D$ has pure dimension $1$.
	We need to prove that every big open embedding
	of (schematic) surfaces is an isomorphism.
	Suppose the contrary, and let $\iota\colon X\into Y$
	be an open embedding such that $Z\de Y\setminus \iota(X)$
	is a non-empty, finite set.
	Let $\Gamma\subseteq X\times Y$ be a graph of $\iota$,
	and let $\ol{\Gamma}\subseteq \ol{X}\times Y$ be the closure of $\Gamma$.
	The projections $p\colon \ol{\Gamma}\to \ol{X}$ and
	$q\colon \ol{\Gamma}\to Y$
	restrict to isomorphisms $p|_{\Gamma}\colon \Gamma \to X$,
	$q|_{\Gamma}\colon \Gamma\to Y\setminus Z=\iota(X)$.
	In particular, $p$ and $q$ are birational.
	
	We claim that $p\colon \ol{\Gamma}\to \ol{X}$ is quasi-finite,
	hence an open embedding by Zariski's Main Theorem \citestacks{082K}.
	Let $F\subseteq \ol{\Gamma}$ be a fibre of $p$.
	Since $p|_{\Gamma}\colon \Gamma \to X$ is an isomorphism,
	the intersection $F\cap \Gamma$ is at most a point.
	Its complement $F\setminus \Gamma$ is mapped by $q$ into the finite set~$Z$.
	Thus, both projections $p,q$ map $F$ to finite sets,
	so $F$ is finite, as claimed.

 	Because $q|_{\Gamma}\colon \Gamma \to Y\setminus Z$
	and $p|_{\Gamma}\colon \Gamma\to X$ are isomorphisms,
	we have $q^{-1}(Z)=\ol{\Gamma}\setminus \Gamma$
	and $D_0\de p(\ol{\Gamma}\setminus \Gamma)=p(\ol{\Gamma})\setminus X$.
	Since $p$ is an open embedding,
	$D_0$ is an open subspace of $\ol{X}\setminus X=D$.
 
 	As $\ol{X}$ is proper, so are the projection $\ol{X}\times Y\to Y$
	and its restriction $q\colon \ol{\Gamma}\to Y$.
	Hence, the preimage $q^{-1}(Z)$ of the finite set $Z$ is proper,
	so its image $D_0=p(q^{-1}(Z))$ is closed in $D$.
	Thus, $D_0$ is a sum of connected components of $D$.
	Then $q\circ p^{-1}\colon p(\ol{\Gamma})\to Y$ is a proper birational morphism
	mapping $D_0$ to the finite set $Z$, 
	so each connected component of $D_0$
	is (schematically) contractible by Lemma \ref{lem:patching}.
	But $D$ has no such connected components, a contradiction.
\end{proof}

\begin{COR}\label{cor:saturation-exists}
	Every surface has a saturation.
	Every schematic surface has a scheme-saturation.
\end{COR}
\begin{proof}
	Let $X$ be a (schematic) surface.
	By the Nagata compactification theorem
	(see \cite{Nagata_schemes} for schemes and
	\cite{Nagata_spaces} for algebraic spaces)
	there is an open embedding $X\into \ol{X}$ into an irreducible,
	proper separated algebraic space (resp.\ a scheme).
	Replacing $\ol{X}$ by its normalization,
	we can assume that $\ol{X}$ is normal.
	Let $D_0$ and $D_1$ be the sum of connected components of
	$\ol{X}\setminus X$ of dimension $0$ and $1$, respectively.
	We can assume that $D_1$ has no (schematically) contractible connected component,
	otherwise we contract it and replace $\ol{X}$ by its image.
	By Proposition~\ref{prop:saturation-criterion},
	the surface $\ol{X}\setminus D_{1}$ is (scheme-)saturated,
	so $X\into \ol{X}\setminus D_1$ is the required saturation.
\end{proof}

\subsection{Uniqueness and universal property}
\label{ssec:saturation_univ-property}

We now prove that the saturation is unique.
To do this, we characterise it by a universal property in the category
$\Dalg$ of surfaces and big open embeddings,
in the same way as \cite[Proposition 2.9]{BodBon5}
characterises scheme-saturation.
As in loc.\ cit., the proof is an easy consequence
of the existence of pushouts (Lemma \ref{lem:push-outs_in_Dalg}).

\begin{PROP}[Universal property of saturated surfaces]\label{prop:univ-property}
	A surface $X$ is saturated if and only if
	for any pair of big open embeddings $i_Y \in \Hom_{\Dalg}(Y,X)$,
	$j \in \Hom_{\Dalg}(Y,Y')$ there exists a unique
	$j'\in \Hom_{\Dalg}(Y', X)$ such that $i_Y = j' \circ j$,
	i.e.\ we have a commutative diagram
	\begin{equation*}\begin{tikzcd}
			Y' \ar[rr,dashed,hook,"\exists !\, j'"] && X. \\
			Y \ar[rru,hook, "i_Y"'] \ar[u,hook,"j"] &&
	\end{tikzcd}\end{equation*}
\end{PROP}
\begin{proof}
	Assume that $X$ satisfies the above universal property.
	To prove that $X$ is saturated,
	we need to prove that every big open embedding $j\in \Hom_{\Dalg}(X,Y')$
	is an isomorphism (Definition \ref{def:saturation}).
	The universal property applied to $Y=X$ and $i_{Y}=\Id_{X}$
	gives a unique $j'\in \Hom_{\Dalg}(Y',X)$
	such that $j' \circ j = \Id_X$.
 	As both $j$ and $j'$ are open embeddings,
	they are isomorphisms, as needed.
	
	Conversely, assume that $X$ is saturated.
	Consider a pair of morphisms $i_Y\in\Hom_{\Dalg}(Y,X)$
	and $j\in \Hom_{\Dalg}(Y,Y')$.
	By Lemma \ref{lem:push-outs_in_Dalg} there exists pushout
	$X'\de X\sqcup_Y Y'$ in $\Dalg$,
	so we have a commutative diagram of big open embeddings
	\begin{equation*}\begin{tikzcd}
			X \ar[r,"f",hook]
			& X' \\
			Y \ar[u,"i_Y",hook] \ar[r,"j",hook]
			& Y'. \ar[u,"g",hook]
	\end{tikzcd}\end{equation*}
	Since $X$ is saturated, $f\colon X\into X'$ is an isomorphism.
	Now $j':=f^{-1}\circ g$ satisfies the required property.
	Such $j'$ is unique as it agrees with $i_Y$
	on a dense open subspace $Y \subseteq Y'$.
\end{proof}

\begin{COR}[Uniqueness of saturation]\label{cor:sat_unique}
	For every surface $X$,
	the saturation $\eta_X \colon X \into X^{\sat}$ is unique,
	up to a unique isomorphism which restricts to the identity on $X$.
	
	The analogous uniqueness result holds for
	scheme-saturation of schematic surfaces.
\end{COR}
\begin{proof}
	Assume that $\wh{\eta} \colon X\into \wh{X}$
	and $\ol{\eta} \colon X\into \ol{X}$ are (scheme-)saturations of $X$.
	By Proposition~\ref{prop:univ-property}
	(or \cite[Proposition 2.9]{BodBon5} for scheme-saturations)
	there exist unique $\ol{j} \colon \wh{X} \into \ol{X}$
	and $\wh{j} \colon \ol{X} \into \wh{X}$ such that
	$\wh{j} \circ \ol{\eta} = \wh{\eta}$
	and $\ol{j} \circ \wh{\eta} = \ol{\eta}$.
	In particular, $\ol{j}$ and $\hat{j}$ restrict to the identity on $X$,
	and satisfy $(\ol{j} \circ \wh{j}) \circ \ol{\eta} = \ol{\eta}$.
	As the image of $\ol{\eta}$ is dense and $\ol{X}$ is separated,
	we conclude that the open embedding $\ol{j} \circ \wh{j}$ is the identity.
	An analogous argument shows that $\wh{j} \circ \ol{j} = \Id_{\wh{X}}$.
\end{proof}

\begin{REM}[In higher dimensions, saturation may not be unique]
	Definition \ref{def:saturation} can be posed
	in the category of normal, irreducible,
	separated algebraic spaces of finite type over $\kk$.
	We can say that $X$ is \emph{saturated}
	if every big open embedding $X\into Y$ in this category is an isomorphism,
	and define saturation as a big open embedding into a saturated space.
	Here \emph{big open} can either mean that the complement has codimension at most $2$;
	or has pure codimension~$2$.
	
	In any case, it is not clear if such a saturation exists,
	and even if it does, it may not be unique,
	in which case the universal property (Proposition \ref{prop:univ-property}) fails.
	The reason for this is the existence of small contractions,
	and subsequent failure of Lemma \ref{lem:push-outs_in_Dalg} on the existence of pushouts.
	
	For a concrete example consider the Atiyah flop, defined as follows,
	see e.g.\ \cite[Example 3-4-3]{Matsuki}.
	Let $Y$ be the quadric cone $\{xy=z_{1}z_{2}\}\subseteq \bA^{4}$,
	and let $X=Y\setminus \{0\}$ be its smooth locus.
	For $i=1,2$ let $X_{i}\to Y$ be the blowup at $D_i\de \{x=z_{i}=0\}$.
	Since $D_{i}\cap X$ is a Cartier divisor on $X$,
	these blowups restrict to isomorphisms over $X$.
	Hence we get open embeddings $X\into X_{i}$
	with complements of pure codimension $2$.
	Each threefold $X_{i}$ is proper over an affine scheme~$Y$,
	hence saturated (cf.\ Proposition \ref{prop:proper-implies-sat}).
	Thus, the embeddings $X\into X_{i}$ are saturations of $X$.
	However, there is no isomorphism $X_{1}\to X_{2}$
	which restricts to the identity on~$X$.
	Indeed, otherwise its composition with the blowup $X_{2}\to Y$
	would agree with the blowup $X_{1}\to Y$ on~$X$,
	hence on $Y$ as $Y$ is separated.
	But the inverse image of $D_{1}$ on $X_{1}$ is a divisor,
	while the inverse image of $D_1$ on $X_2$ has the exceptional $\bP^1$
	as its irreducible component, a contradiction.
	
	Note that Lemma \ref{lem:push-outs_in_Dalg} fails in this example, too.
	Indeed, suppose there is a separated algebraic space
	$Z\de X_1\sqcup_{X} X_2$
	with big open embeddings $X_{i}\into Z$ which agree on $X$.
	Since each $X_{i}$ is saturated,
	each $X_{i}\into Z$ must be an isomorphism,
	so composing them we get an isomorphism $X_{1}\to X_{2}$
	which is an identity on $X$,
	but we have shown above that this is impossible.
\end{REM}

\subsection{Saturation as the adjunction unit}\label{ssec:adjunction}

We now use the universal property (Proposition \ref{prop:univ-property})
to characterise of the saturation morphism $X\into X^{\sat}$
in categorical terms, as the adjunction unit for the localisation 
$\Dalg\to \wh{\Dalg}$ in the set of all morphisms (Theorem \ref{thm:adjoint}).
The analogous result for scheme-saturation
was proved in \cite[Theorem 2.10]{BodBon5}.
Like in the previous section,
our proof relies on the existence of pushouts 
(Lemma \ref{lem:push-outs_in_Dalg}) and formal arguments 
along the lines of \cite[\S 2.4]{BodBon5}.
Nonetheless, for the readers' convenience we present the argument in detail.

To define the localisation $\wh{\Dalg}$
we need the following analogue of \cite[Proposition 2.3]{BodBon5}.

\begin{PROP}\label{prop_LMS}
	All morphisms in $\Dalg$ form a left multiplicative system.
\end{PROP}
\begin{proof}
	We need to check conditions (LMS 1)--(LMS 3) from \citestacks{04VC}. 
	Condition (LMS 1) states that the given class of morphisms
	contains the identity morphisms and is closed under composition.
	As we consider all morphisms in $\Dalg$, it is clearly satisfied. 
	The second condition (LMS 2) states that for any 
	$j_1 \in \Hom_{\Dalg}(X_0, X_1)$ and $j_2 \in \Hom_{\Dalg}(X_0, X_2)$
	there exists $X\in \Dalg$ and $i_1\colon X_1 \to X$
	and $i_2 \colon X_2\to X$ such that $i_1\circ j_1 = i_2 \circ j_2$.
	This holds by the pushout Lemma \ref{lem:push-outs_in_Dalg}. 
	The last condition (LMS 3) states that for any morphisms 
	$f,g,j$ in $\Dalg$ such that $f\circ j=g\circ j$
	there exists a morphism $i$ in $\Dalg$ such that $i\circ f=i\circ g$.
	Since the image of $j$ is dense, the condition $f\circ j=g\circ j$
	implies that $f=g$, and we can take $i=\Id$.
\end{proof}

The objects of $\wh{\Dalg}$ are the objects of $\Dalg$,
while morphisms are equivalence classes of \emph{roofs}
\begin{equation*}\begin{tikzcd} [row sep=small]
	&Z&\\
	X\ar[ur,"f"] & & Y, \ar[ul,"s"']
\end{tikzcd}\end{equation*} 
where $f,s$ are morphisms in $\Dalg$,
see \citestacks{04VB} for details.
We denote such a roof by $s^{-1}f \colon X\to Y$.
We denote by $\mathfrak{s}\colon \Dalg \to \wh{\Dalg}$
the canonical functor which is the identity on objects and 
maps $f\in \Hom_{\Dalg}(X,Y)$ to the equivalence class of the roof
$\Id_Y^{-1}f\in \Hom_{\wh{\Dalg}}(X,Y)$.

By Corollary \ref{cor:saturation-exists}, 
every surface $X$ admits a saturation $\eta_{X}\colon X\into X^{\sat}$.
We define a functor $\mathfrak{s}_* \colon \wh{\Dalg} \to \Dalg$
mapping a surface $X$ to $X^{\sat}$,
and a roof $s^{-1}f \in \Hom_{\wh{\Dalg}}(X,Y)$
to a morphism $\mathfrak{s}_{*}(s^{-1}f)\in \Hom_{\Dalg}(X^{\sat},Y^{\sat})$
defined as follows.
By Proposition \ref{prop:univ-property},
we have a unique morphism $i\colon Z\to X^{\sat}$
such that $i \circ f = \eta_X$, and a unique morphism
$\mathfrak{s}_{*}(s^{-1}f)\colon X^{\sat}\to Y^{\sat}$
such that $\mathfrak{s}_{*}(s^{-1}f)\circ (i\circ s)\circ=\eta_{Y}$,
i.e.\ the following diagram commutes 
\begin{equation}\label{eqtn_s_on_Hom}
	\begin{tikzcd} [row sep=small]
		X^{\sat} \ar[rr,"\mathfrak{s}_*(s^{-1}f)"] & & Y^{\sat} \\
			& Z \ar[ul,"i"] & \\
			X \ar[ur,"f"'] \ar[uu,"\eta_X"] & & Y. \ar[uu,"\eta_Y"'] \ar[ul,"s"]
		\end{tikzcd} 	
\end{equation}
One checks that $\mathfrak{s}_*(s^{-1}f)$
depends only on the equivalence class of a roof $s^{-1}f$,
and $\mathfrak{s}_{*}$ respects the composition, cf.\ \cite[\S 2.4]{BodBon5}.
Hence, we get a functor $\mathfrak{s}_* \colon \wh{\Dalg} \to \Dalg$.

\begin{THM}[{cf.\ \cite[Theorem 2.10]{BodBon5}}]
\label{thm:adjoint}
	The functor $\mathfrak{s}_*$ is left adjoint to $\mathfrak{s}$.
	Moreover, for any surface $X\in \Dalg$, the saturation morphism 
	$\eta_X \colon X\to X^{\sat} =\mathfrak{s}_*\mathfrak{s}(X)$
	is the adjunction unit.
\end{THM}
\begin{proof}
	To show that $\mathfrak{s}_*$ is left adjoint to $\mathfrak{s}$,
    	we need to check that for any $X, Y \in \Dalg$, the map 
	\begin{equation*} 
		\alpha \colon \Hom_{\wh{\Dalg}}(X, Y) \to 
		\Hom_{\Dalg}(X, Y^{\sat}),\qquad \alpha \colon s^{-1}f 
			\mapsto \mathfrak{s}_{*}(s^{-1}f)\circ \eta_X,
		\end{equation*} 
	is a bijection.
    	Since $\alpha$ maps the identity on $X$ to $\eta_{X}$,
    	this claim implies that $\eta_{X}$ is the adjunction unit.
	In fact, we claim that the inverse of $\alpha$ is given by 
	\begin{equation*} 
			\beta\colon \Hom_{\Dalg}(X, Y^{\sat})
			\to \Hom_{\wh{\Dalg}}(X, Y),
			\qquad \beta\colon h \mapsto \eta_{Y}^{-1}h.
		\end{equation*}
	We have $\beta(\alpha(s^{-1}f)) = \eta_Y^{-1}(\mathfrak{s}_{*}(s^{-1}f) \circ \eta_X)$ 
    	which is equivalent to $s^{-1}f$.
    	Indeed, considering $u\de \mathfrak{s}_{*}(s^{-1}f) \circ \eta_X$,
    	$i\colon Z\to X^{\sat}$ as in \eqref{eqtn_s_on_Hom},
    	and $v\de \mathfrak{s}_{*}(s^{-1}f) \circ i$,
	diagram \eqref{eqtn_s_on_Hom}
	shows that $v\circ f=u$ and $v\circ s=\eta_{Y}$.
	The required equivalence $\eta_{Y}^{-1}u=s^{-1}f$
    	is witnessed by the commutative diagram
		\begin{equation*} 
		    \begin{tikzcd} 
			& Y^{\sat} \ar[d,"\Id" description] &\\
			X \ar[ur,"u" description]  \ar[r,"u" description] \ar[dr,"f" description] &
			Y^{\sat} & 
			Y \ar[ul,"\eta_Y" description]
			\ar[l,"\eta_Y" description]
			\ar[dl,"s" description] \\ & 
			Z.\ar[u,"v" description]&
		    \end{tikzcd} 
		\end{equation*} 
	In the other direction, 
    	$\alpha(\beta(h)) = \alpha(\eta_Y^{-1}h) =
    	\mathfrak{s}_{*}(\eta_{Y}^{-1}h)\circ \eta_{X}$.
	Diagram \eqref{eqtn_s_on_Hom} for the roof $\eta_{Y}^{-1}h$ is  
	\begin{equation*}
		\begin{tikzcd}[row sep=small]
			X^{\sat} \ar[rr,"\mathfrak{s}_*(\eta_Y^{-1}h)"] & & Y^{\sat} \\
			& Y^{\sat} \ar[ul,"i"] & \\
			X \ar[ur,"h"'] \ar[uu,"\eta_X"] & & Y, \ar[uu,"\eta_Y"'] \ar[ul,"\eta_Y"]
		\end{tikzcd} 	
	\end{equation*}
	We have $\mathfrak{s}_*(\eta_Y^{-1}h)\circ i=\Id_{Y^{\sat}}$
    	because these morphisms agree on a dense open subspace
	$Y\subseteq X^{\sat}$. Hence $\alpha(\beta(h)) =
    	\mathfrak{s}_{*}(\eta_{Y}^{-1}h)\circ \eta_{X} = 
	\mathfrak{s}_{*}(\eta_{Y}^{-1}h)\circ i\circ h=h$, as needed.
\end{proof}

\begin{PROP}[{cf.\ \cite[Proposition 2.11]{BodBon5}}] 
\label{prop:fully-faithful}
	The functor $\mathfrak{s}_*$ is fully faithful.	
\end{PROP}
\begin{proof}
	By \cite[Corollary 1.23]{Huy}, it suffices to check that the adjunction counit
	$\mathfrak{s}\mathfrak{s}_* \to \Id$ is an isomorphism when applied to any
	object in $\wh{\Dalg}$. This is clear as $\wh{\Dalg}$ is a groupoid.
\end{proof}

\subsection{Saturation of the schematic locus}
\label{ssec:schematic-locus} 
Recall that by Corollaries \ref{cor:saturation-exists} and \ref{cor:sat_unique},
every surface $X$ has a unique saturation $X\into X^{\sat}$;
and every schematic surface has a unique scheme-saturation
$X\into X^{\schsat}$. The following observation relates those two. 

\begin{PROP}
\label{prop:sat_and_sch_commute}
	For every surface $X$ we have a natural isomorphism 
	$(X^{\sat})^{\sch}\cong (X^{\sch})^{\schsat}$
	which restricts to the identity on $X^{\sch}$.
	In other words, the schematic locus of the saturation of $X$
	is the scheme-saturation of the schematic locus of $X$.
\end{PROP}
\begin{proof}
	By uniqueness of scheme-saturation (Corollary \ref{cor:sat_unique}),
	it is enough to check that the schematic locus  $(X^{\sat})^{\sch}$
	is scheme-saturated and contains $X^{\sch}$ as a big open subscheme.
	
	As every point of the schematic locus $X^{\sch}$
	admits an affine open neighbourhood, 
	the big open embedding 
	$X^{\sch} \into X \into X^{\sat}$
	factors through $(X^{\sat})^{\sch}$,
	so $(X^{\sat})^{\sch}$ contains $X^{\sch}$ as a big open subscheme. 

	To prove that $(X^{\sat})^{\sch}$ is scheme-saturated,
	we need to prove that every big open embedding 
	$j\colon (X^{\sat})^{\sch}\into Y$ of schematic surfaces is an isomorphism. 
	By Lemma \ref{lem:push-outs_in_Dalg}
	we have a surface $\ol{Y}\de X^{\sat}\sqcup_{(X^{\sat})^{\sch}} Y$
	and a commutative square of big open embeddings 
	\begin{equation*}
		\begin{tikzcd}
			X^{\sat} \ar[r,"\ol{j}"] & \ol{Y} \\
			(X^{\sat})^{\sch} \ar[r,"j"] \ar[u,"i"] & Y.
			\ar[u,"\ol{i}"] \ar[ul,"i'" description] \ar[l,"\wt{i}",bend left=30]
		\end{tikzcd} 
	\end{equation*}
	Since $X^{\sat}$ is saturated, 
	$\ol{j}\colon X^{\sat}\into Y$ is an isomorphism.
 	Hence $i'\de \bar{j}^{-1}\circ \bar{i}\colon Y\into X^{\sat}$
	is an open embedding.
	Since every point in $Y$ has an affine neighbourhood,
	$i'$ factors through $(X^{\sat})^{\sch}$, i.e.
	we have an open embedding  $\wt{i}\colon Y \into (X^{\sat})^{\sch}$
	such that $i' = i \circ \wt{i}$.
	Thus, $i\circ \wt{i}\circ j=i'\circ j=\bar{j}^{-1}\circ \bar{i}\circ j=i$.
	As $i$ is an embedding, we get that $\wt{i} \circ j$
	is the identity on $(X^{\sat})^{\sch}$.
	As both $\wt{i}$ and $j$ are open embeddings,
	we conclude that $j$ is an isomorphism, as needed.
\end{proof}

\subsection{Saturation is functorial for proper surjective morphisms}
\label{ssec:functoriality}

\begin{THM}
\label{thm:functoriality}
	Let $\pi\colon Y\to X$ be a proper surjective morphism of surfaces.
	Then 
	\begin{enumerate}
		\item\label{item:functoriality-iff}
		 	the surface $X$ is saturated if and only if $Y$ is saturated, and
		\item\label{item:functoriality-diagram}
			there is a unique proper surjective morphism 
			$\pi^{\sat}\colon Y^{\sat}\to X^{\sat}$ 
			such that the following diagram is commutative
			\begin{equation*}
				\begin{tikzcd}
					Y \ar[r,hook,"\eta_{Y}"] \ar[d,"\pi"] &
					Y^{\sat} \ar[d,"\pi^{\sat}"] \\
					X \ar[r,hook,"\eta_{X}"] & X^{\sat}.
				\end{tikzcd}
			\end{equation*}
	\end{enumerate}
\end{THM}

Applying Stein factorisation \citestacks{0A1B}
we see that it is enough to prove
Theorem \ref{thm:functoriality} in two cases:
when $\pi$ is birational and when $\pi$ is finite.
The birational case is easy and holds for scheme-saturation, too:
we settle it in Proposition \ref{prop:functoriality-birational} below.
We begin with the finite case,
which is more involved and fails for scheme-saturation,
as shown by Example \ref{ex:covers}.

\begin{PROP}
	\label{prop:covers}
	Theorem \ref{thm:functoriality} holds if
	the proper surjective morphism $\pi\colon Y\to X$ is finite.
\end{PROP}
\begin{proof}
	The uniqueness statement in \ref{item:functoriality-diagram}
	follows from the fact that
	that $Y$ is a dense open subspace of $Y^{\sat}$ and $X^{\sat}$ is separated.
	We show \ref{item:functoriality-iff}
	and the existence part of \ref{item:functoriality-diagram}.
	
	Factoring $Y\to X$ via the integral closure of $X$
	in the maximal separable subfield of $\kappa(Y)/\kappa(X)$,
	we see that it suffices to consider two cases:
	first, when $\kappa(Y)/\kappa(X)$ is purely inseparable,
	and second, when it is separable, i.e.\ $Y\to X$ is a cover
	(see Section \ref{ssec:covers}).
	
	We start with the purely inseparable case.
	Then some power $F^n_X$ of the absolute Frobenius $F_X:X\to X$
	dominates $Y\to X$, that is, we have a factorisation
	$F^n_X:X\to Y \to X$.
	Let $Y'$ (respectively, $X'$) be the integral closure of
	$X^{\sat}$ in $\kappa(Y)$
	(respectively, in the field extension $F^n_X:\kappa(X)\to\kappa(X)$).
	By Lemma \ref{lem:G-lifting},
	the morphisms $Y\into Y'$ and $X\into X'$ are big open embeddings.
	We remark that the absolute Frobenius $X\to X$
	is not a morphism over $\Spec \kk$, nonetheless,
	the proof of Lemma \ref{lem:G-lifting} applies in this case, too.
	
	By the universal property of the saturation (Proposition \ref{prop:univ-property})
	we obtain a big open embedding $X'\into (X')^{\sat}=X^{\sat}$.
	We have now constructed a commutative diagram
	\begin{equation*}
		\begin{tikzcd}
			X \ar[r,hook] \ar[d] \ar[dd,"F^{n}_{X}"', bend right=90] &
			X' \ar[d] \ar[r,"\eta_{X'}",hook] & X^{\sat} \ar[ldd,"F_{X^{\sat}}^{n}"] \\
			Y \ar[r,hook] \ar[d] & Y' \ar[d] & \\
			X  \ar[r,hook] & X^{\sat}, &
		\end{tikzcd}
	\end{equation*}
	where the horizontal morphisms are big open embeddings
	and the vertical ones are surjective.
	Since $X'\to X^{\sat}$ is surjective and
	$F_{X^{\sat}}^{n}\colon X^{\sat}\to X^{\sat}$ is a homeomorphism,
	we conclude that the open embedding $\eta_{X'}\colon X'\into X^{\sat}$
	is surjective, so it is an isomorphism.
	
	Let $X''$ be the integral closure of $Y^{\sat}=(Y')^{\sat}$
	in the function field extension $\kappa(X')/\kappa(Y)$.
	As before, Lemma \ref{lem:G-lifting} shows that $X'$
	is a big open subspace of $X''$.
	Since $X'=X^{\sat}$ is saturated, we get $X''=X'$.
	Further, the morphism $X''\to Y^{\sat}$ is surjective and
	we conclude that $Y'\into Y^{\sat}$
	is surjective as well, and thus $Y'=Y^{\sat}$.
	This proves \ref{item:functoriality-diagram}.

	To see part \ref{item:functoriality-iff},
	recall that $X'=X^{\sat}$ and $Y'=Y^{\sat}$,
	so if $X=X^{\sat}$ (resp.\ $Y=Y^{\sat}$),
	then the equality $Y=Y^{\sat}$ (resp.\ $X=X^{\sat}$)
	follows from the commutativity of the upper (resp.\ lower)
	square in the diagram above.
	
	We continue with the case where $Y\to X$ is a cover.
	First, we prove part \ref{item:functoriality-iff}.
	Replacing $Y\to X$ by the integral closure $Y''$ of $X$ in the Galois
	closure of $\kappa(Y)/\kappa(X)$,
	we can assume that the cover $Y\to X$ is Galois.
	Indeed, assume that \ref{item:functoriality-iff} holds for Galois covers.
	As $Y''\to X$ and $Y''\to Y$ are both Galois covers,
	we obtain that $X$ is saturated if and only if $Y''$ is saturated,
	which holds if and only if $Y$ is saturated, as needed.
	
	Let $G$ be the Galois group of $Y\to X$.
	By Lemma \ref{lem:lifting-Galois}\ref{item:G-acts},
	$G$ acts on $Y$ and $X=Y \git G$.
	
	Assume that $Y$ is saturated.
	Consider the integral closure $Y'$ of $X^{\sat}$ in $\kappa(Y)$.
	By Lemma \ref{lem:G-lifting},
	the induced morphism $Y\to Y'$ is a $G$-equivariant big open embedding.
	Since $Y$ is saturated, we have $Y=Y'$,
	hence $X=X^{\sat}$, as needed.
	
	Conversely, assume that $X$ is saturated.
	Since saturation is unique up to a unique isomorphism,
	see Corollary \ref{cor:sat_unique},
	the $G$-action on $Y$ extends to a $G$-action on $Y^{\sat}$.
	By Keel--Mori \cite{keel-mori} we can form the quotient
	$X':=Y^{\sat}\git G$ in the category of algebraic spaces.
	Since $Y^{\sat}$ is normal and separated, so is $X'$.

	By the universal property of the quotient $Y\to X$
	we obtain a natural morphism $X\to X'$ such that the diagram
	\begin{equation}
		\label{diagram:saturation}
		\begin{tikzcd}
			Y \ar[r,hook] \ar[d] & Y^{\sat}\ar[d]\\
			X \ar[r] & X'
		\end{tikzcd}
	\end{equation}
	commutes. We claim that $X\to X'$ is a big open embedding.
	As fibres of $Y^{\sat}\to X'$ correspond to $G$-orbits
	and $Y^{\sat}\setminus Y$ consists of finitely many points,
	we find that $X\to X'$ has finite fibres and that
	the image of $X$ misses only finitely many points in $X'$.
	Furthermore, $X\to X'$ is birational as
	$\kappa(X)=\kappa(Y)^{G}=\kappa(Y^{\sat})^{G}=\kappa(X')$.
	By Zariski's Main Theorem \citestacks{082K}
	we conclude that the quasi-finite birational morphism
	$X\to X'$ is an open embedding, as claimed.
	Since $X$ is saturated, we get $X=X'$, so $Y=Y^{\sat}$, as needed.
	
	It remains to show \ref{item:functoriality-diagram}, that is,
	to construct a morphism $Y^{\sat}\to X^{\sat}$.
	We first do so in the Galois case.
	Note that we can form the diagram \eqref{diagram:saturation}
	without the saturatedness assumption on $X$.
	Then we conclude from \ref{item:functoriality-iff}
	that $X'=X^{\sat}$, so we have found the desired morphism.
	
	In the general case, consider the integral closure $Z$
	of $Y$ in the Galois closure of the function field extension
	$\kappa(Y)/\kappa(X)$.
	Note that $Z\to Y$ and $Z\to X$ are Galois covers with
	Galois groups $H$ and $G$ respectively. Further, note
	that $H\leq G$, $Y=Z\git  H$, and $X=Z\git  G$.
	By the above discussion we also have that $Y^{\sat}=Z^{\sat}\git H$
	and $X^{\sat}=Z^{\sat} \git  G$.
	The morphism $Y^{\sat}\to X^{\sat}$
	is obtained by factoring the morphism $Z^{\sat}\to X^{\sat}$
	via the $H$-quotient $Z^{\sat} \git  H=Y^{\sat}$.
\end{proof}

\begin{EXM}[Theorem \ref{thm:functoriality}\ref{item:functoriality-iff}
	fails for scheme-saturation]
	\label{ex:covers}
	There exist covers $Y\to X$ of smooth surfaces
	such that $X$ is scheme-saturated, but $Y$ is not.
	Indeed, let $X$ be any smooth,
	scheme-saturated surface which is not saturated,
	see Section \ref{ssec:examples-scheme-sat} for concrete ones.
	By Lemma~\ref{lem:Galois},
	there is a cover $g\colon Y'\to X^{\sat}$
	such that $Y'$ is schematic.
	Put $Y=g^{-1}(X)$.
	Then $g|_{Y}\colon Y\to X$ is a cover
	and $Y\into Y'$ is an open embedding of schematic surfaces.
	The complement $Y'\setminus Y$ is the preimage
	of a finite set $X^{\sat}\setminus X$,
	which is non-empty because $X$ is not saturated.
	Hence, $Y'\setminus Y$ is finite and non-empty,
	so $Y$ is not scheme-saturated, as claimed.
\end{EXM}

\begin{PROP}
	\label{prop:functoriality-birational}
	Theorem \ref{thm:functoriality} holds
	if the proper morphism $\pi\colon Y\to X$ is birational.
	
	Moreover, if $X$ and $Y$ are schematic,
	then the analogous statement holds for scheme-saturation.
\end{PROP}
\begin{proof}
	We give a proof for saturation,
	for scheme-saturation the argument is exactly the same.
	As before, we note that the uniqueness of
	$\pi^{\sat}$ in \ref{item:functoriality-diagram}
	follows from the fact that $X^{\sat}$ is separated
	and $Y$ is a dense open subspace of $Y^{\sat}$.
	We now prove the existence part of \ref{item:functoriality-diagram}.

	Let $\ol{Y}$ be a Nagata compactification of $Y$ over $X^{\sat}$,
	that is, we have a commutative diagram
	\begin{equation*} 
		\begin{tikzcd}
			Y \ar[r,hook] \ar[d,"\pi"] &
			\ol{Y} \ar[d,"\ol{\pi}"] \\
			X  \ar[r,hook,"\eta_{X}"] & X^{\sat},
		\end{tikzcd}
	\end{equation*}
	where $Y\into \ol{Y}$ is an open embedding
	and $\ol{\pi}\colon \ol{Y}\to X^{\sat}$ is proper.
	Replacing $\ol{Y}$ by its normalisation we can assume that it is normal.

	Since $\pi$ is birational, so is $\ol{\pi}$.
	Let $E\subseteq \ol{Y}$ be the exceptional divisor of $\ol{\pi}$.
	Its open subspace $E_0\de E\cap Y$ is the exceptional locus of $\pi$.
	Since $\pi$ is proper, so is $E_0$.
	In particular, $E_0$ is closed in $E$,
	so $E_0$ is a sum of connected components of $E$.
	Since $E$ is contractible, by Lemma \ref{lem:patching}, so is $E-E_{0}$.
	Contracting it and replacing $\ol{Y}$ by its image
 	we can assume that $E\subseteq Y$.
	In particular, $\ol{Y}\setminus Y\cong X^{\sat}\setminus X$
	is a finite set, so $Y\into \ol{Y}$ is a big open embedding.
	We claim that it is the saturation of $Y$.
	By Definition \ref{def:saturation} we need to check
	that every big open embedding
	$\ol{Y}\into \ol{Y}'$ is an isomorphism.
	Let $U=\ol{Y}\setminus E$, 
	$U'=\ol{Y}'\setminus E$.
	Then $U\into U'$ and
	$\ol{\pi}|_{U}\colon U\into X^{\sat}$ are big open embeddings. 
	By Lemma \ref{lem:push-outs_in_Dalg},
	we have a big open embedding 
	$X^{\sat}\into U'\sqcup_{U} X^{\sat}$,
	which is an isomorphism because $X^{\sat}$ is saturated. 
	Hence, $U'\into X^{\sat}$ glues with 
	$\ol{\pi}\colon \ol{Y}\to X^{\sat}$
	to a morphism $\ol{Y}'\to X^{\sat}$. 
	Since $\ol{\pi}\colon \ol{Y}\to X^{\sat}$ is proper,
	so is the embedding $\ol{Y}\into \ol{Y}'$ by \citestacks{04NX},
	hence $\ol{Y}\into \ol{Y}'$ is an isomorphism, as needed. 

	It remains to prove \ref{item:functoriality-iff}.
	If $X$ or $Y$ is saturated,
	then $\eta_{X}$ or $\eta_{Y}$ is an isomorphism,
	so the composition 
	$\eta_{X}\circ \pi=\pi^{\sat}\circ \eta_{Y}$ is proper.
	It follows that $\eta_{Y}$ is proper and,
	since $\pi$ is surjective, $\eta_{X}$ is proper, too by
	\cite[{\href{https://stacks.math.columbia.edu/tag/04NX}{04NX},
	\href{https://stacks.math.columbia.edu/tag/0AGD}{0AGD}}]{stacks-project}.
	Hence $\eta_{X}$ and $\eta_{Y}$ are isomorphisms, as needed.
\end{proof} 	

\begin{proof}[Proof of Theorem \ref{thm:functoriality}]
	Applying Stein factorisation 
	\cite[\href{https://stacks.math.columbia.edu/tag/0A1B}{Tag 0A1B}]{stacks-project},
	we can assume that $\pi$ is finite or birational.
	These cases are settled in Propositions \ref{prop:covers}
	and \ref{prop:functoriality-birational}, respectively.
\end{proof}

\subsection{Reconstruction of the saturation from the category of reflexive sheaves}
\label{ssec:sat_from_ref}

In this section we prove that
the category $\Refll(X)$ of reflexive sheaves
on $X$ determines the surface model $X^{\sat}$ of $X$.
This is the analogue of the main theorem of \cite{BodBon5},
which asserts the same for $X^{\schsat}$.
Furthermore, we prove that given $X$,
the saturation morphism $\eta_X\colon X\to X^{\sat}$
can be reconstructed from $\Refll(X)$, too,
which improves the results of \cite{BodBon5}. 
Before we state this result in a precise way in
Theorem \ref{thm:sat_from_Ref} below, we need some preparations.

\begin{PROP}
\label{prop:Ref_gives_sat}
	Let $X$ be a surface,
	and let $\ol{\Refll(X)}$ be the geometric model
	of the category of reflexive sheaves on $X$,
	see \cite[Formula (23)]{BodBon5}.
	Then the following hold.
	\begin{enumerate}
		\item \label{item:Ref-big-opens}
			There are big open embeddings $X^{\sch}\into X$
			and $X^{\sch}\into \ol{\Refll(X)}$.
		\item \label{item:Ref-saturations}
			We have an isomorphism
			$\ol{\Refll(X)}^{\sat}\cong X^{\sat}$.
	\end{enumerate}
\end{PROP}

\begin{proof}
	Since the schematic locus $X^{\sch}$ of $X$ is a big open subspace,
	we have $\Refll(X) \cong \Refll(X^{\sch})$.
	By \cite[Corollary 5.6]{BodBon5},
	$\ol{\Refll(X^{\sch})}$ is isomorphic to $(X^{\sch})^{\schsat}$,
	which contains $X^{\sch}$ as a big open subspace.
	This proves \ref{item:Ref-big-opens}.
	Part \ref{item:Ref-saturations} follows because big open subspaces
	have the same saturation, see Proposition \ref{prop:univ-property}.
\end{proof}

The category $\Refll(X)$ is quasi-abelian (cf. \cite[Theorem 4.3]{BodBon5}), 
hence it admits the canonical exact structure,
see \cite[Corollary 3.14]{BodBon5}.
By \cite[Corollary 4.4]{BodBon4} the right abelian envelope
$\cA_r(\Refll(X))$, see \cite[Definition 4.2]{BodBon4}, 
is equivalent to $\Coh^{\geq 1}(X):= \Coh(X)/\Coh_0(X)$
which is the quotient of the category $\Coh(X)$ of
coherent sheaves on $X$ by the Serre subcategory $\Coh_0(X)$
of sheaves with zero dimensional support. 

Proposition \ref{prop:Ref_gives_sat}\ref{item:Ref-big-opens}
implies that the surfaces $X$ and $\ol{\Refll(X)}$ share a big open subspace,
so their categories of reflexive sheaves are equivalent.
Fix an equivalence 
\begin{equation*}
	\varphi\colon \Refll(X^{\sch}) \to \Refll(\ol{\Refll(X)})).
\end{equation*}
By \cite[Proposition 4.9]{BodBon4},
the functor $\varphi$ induces an
equivalence of the right abelian envelopes
$\cA_r(\varphi) \colon \Coh^{\geq 1}(X^{\sch}) \to \Coh^{\geq 1}(\ol{\Refll(X)})$.
In view of \cite[Proposition 2.7]{CalPir},
the morphism $\cA_r(\varphi)$ induces an isomorphism of
locally ringed spaces 
$\psi \colon (X^{\sch})^{\geq 1} 
\xrightarrow{\simeq} \ol{\Refll(X)}^{\geq 1}$,
where for a scheme $Y$,
$i \colon Y^{\geq 1} \into Y$
is the subset of points of dimension at
least one with the subspace topology and the sheaf $i^{-1} \cO_Y$,
see \cite[\S 2]{CalPir}.
Then, by \cite[Proposition 3.1]{CalPir}, 
the isomorphism $\psi$ gives rise to big open subschemes 
$U \subset X^{\sch}$, $V \subset \ol{\Refll(X)}$,
and an isomorphism $f\colon U \xrightarrow{\simeq} V$
which restricts to $\psi$.
We can thus view $f$ as a big open embedding
$f' \colon U \into \ol{\Refll(X)}$ of a big open subscheme
$U \subset X^{\sch}$.
Composing $f'$ with the saturation
$\ol{\Refll(X)}\into \ol{\Refll(X)}^{\sat}$ yields a big open embedding 
$j \colon U \into \ol{\Refll(X)}^{\sat}$ of a big open subspace 
$U \subset X^{\sch} \subset X$ 
into a saturated surface.
By the universal property of saturation 
(Proposition \ref{prop:univ-property}),
$j$ induces a unique morphism 
\begin{equation}\label{eqtn:ol_j}
	\overline{j} \colon X\to \ol{\Refll(X)}^{\sat}. 	
\end{equation}

\begin{THM}
	\label{thm:sat_from_Ref}
	Let $X$ be a surface.
	Under the isomorphism $\ol{\Refll(X)}^{\sat}\cong X^{\sat}$
	of Proposition \ref{prop:Ref_gives_sat}\ref{item:Ref-saturations},
	the morphism \eqref{eqtn:ol_j}
	corresponds to the saturation $\eta_X \colon X\to X^{\sat}$.
\end{THM}

\begin{proof}
	Since $\ol{\Refll(X)}^{\sat}$ is saturated,
	the big open embedding \eqref{eqtn:ol_j} of $X$
	into $\ol{\Refll(X)}^{\sat}$ is the saturation of $X$.
	The latter is unique up to 
	a unique isomorphism by Corollary \ref{cor:sat_unique}.
\end{proof}

\section{Affinisation of saturated surfaces}
\label{sec:affinisation}

In this section we prove Theorem \ref{thm:intro_sat_implies_proper},
which implies that a saturated surface is proper over its affinisation
provided the latter is non-trivial.
To do this, we start by recalling some 
known properties of affinisation (Section \ref{ssec:def_of_aff})
and of the intersection form on singular surfaces 
(Section \ref{ssec:intersection form}),
which are frequently used in the remaining part of the article.
The proof of Theorem \ref{thm:intro_sat_implies_proper}
is given in Section \ref{ssec:prop_of_aff}.
Next, in Section \ref{ssec:surj_of_aff}
we comment that the affinisation morphism from 
a saturated surface is always surjective 
(Proposition \ref{prop:surjective-2}) 
and coincides with the Iitaka fibration
 associated to the boundary divisor (Remark \ref{rem:Iitaka}). 

\subsection{The affinisation of surfaces}\label{ssec:def_of_aff}

We now recall basic properties of the affinisation.
For more details in the setting of locally ringed spaces,
we refer to \cite[Proposition 3.4]{GorWed}.

The \emph{affinisation functor} maps a surface $X$
to its \emph{affinisation} $X^{\aff} := \Spec H^0(X, \cO_X)$,
and a morphism $f\colon X\to Y$ of surfaces to a morphism
$f^{\aff} \colon X^{\aff} \to Y^{\aff}$
of affine schemes induced by the morphism 
$f^{\sharp} \colon \cO_Y \to f_* \cO_X$ of sheaves on $Y$. 
By \citestacks{081X}, every surface has an 
\emph{affinisation morphism} $\alpha_{X}\colon X\to X^{\aff}$
which satisfies the following universal property:
every morphism from $X$ to an affine scheme factors uniquely through $\alpha_{X}$.
It follows that for every morphism of surfaces $f\colon X\to Y$,
the morphisms $\alpha_Y \circ f$ and $f^{\aff} \circ \alpha_X$
from $X$ to $Y^{\aff}$ coincide, that is,
the affinisation morphism induces a natural transformation. 

Let $X$ be a surface.
By \cite[Corollary 6.2]{Schroer} the ring $H^{0}(X,\cO_{X})$
is a finitely generated algebra of dimension at most $2$,
hence $X^{\aff}$ is an affine variety.
We note that this no longer holds in higher dimension,
see \cite[Example 3.2.3]{Brion_group-notes} or \cite{Vakil-non-fg}
for an explicit example of dimension~$3$.

Since $X$ is normal,
by the universal property of the normalisation so is $X^{\aff}$. 
If $\iota\colon U\into X$ is a big open embedding,
then $\iota^{\aff}\colon U^{\aff}\to X^{\aff}$ is an isomorphism
by Hartogs's lemma. 
Hence the affinisations of $U$ and $X$ coincide,
i.e., $\alpha_{U}=\alpha_{X}\circ\iota$. 
In particular, the affinisation of $X$ coincides 
with the affinisation of its schematic locus $X^{\sch}$,
and of its saturated model $X^{\sat}$.

\begin{LEM}\label{lem:surjective-1}
	Let $X$ be a surface.
	Then the affinisation morphism $\alpha_X\colon X\to X^{\aff}$ is dominant.
	Moreover, if $\dim X^{\aff}\leq 1$,
	then $\alpha_X$ is surjective.
\end{LEM}
\begin{proof}
	If $\alpha_{X}$ is not dominant,
	then its image is contained in a 
	proper closed subspace $Z$ of $X^{\aff}$,
	which is affine since $X^{\aff}$ is;
 	contrary to the universal property of $X^{\aff}$.
	If $\dim X^{\aff}\leq 1$,
	then since $\alpha_{X}$ is dominant, its image is affine,
	hence equal to $X^{\aff}$ by the universal property.
\end{proof}

\begin{LEM}\label{lem:resolution-aff}
	Let $\pi\colon Y\to X$ be a proper birational morphism of surfaces. 
	Then the composition $\alpha_{X}\circ \pi\colon Y\to X^{\aff}$
	is the affinisation morphism $\alpha_Y$ of $Y$.
	In other words, $\pi^{\aff}$ is the identity.
\end{LEM}
\begin{proof}
	As geometric fibres of a birational morphism are connected,
	\cite[\href{https://stacks.math.columbia.edu/tag/0A1B}{Tag 0A1B}]{stacks-project}
	implies that $\pi_* \cO_Y \simeq \cO_X$.
	Hence, $\pi^{\aff}$ is the identity morphism.
	Since the affinisation morphism is a natural transformation,
	we conclude that $\alpha_Y \circ \pi = \pi^{\aff} \circ \alpha_X = \alpha_X$.
\end{proof}

\begin{PROP}\label{prop:proper-surj-aff}
	Let $\pi\colon Y\to X$ be a proper surjective morphism of surfaces.
	Then the induced morphism 
	$\pi^{\aff}\colon Y^{\aff}\to X^{\aff}$ is finite and surjective.
	In particular, $\dim Y^{\aff}=\dim X^{\aff}$.
\end{PROP}
\begin{proof}
	Applying Stein factorisation and Lemma \ref{lem:resolution-aff},
	we can assume that $\pi$ is finite.
	
	By Lemma \ref{lem:surjective-1}, the affinisation morphism
	$\alpha_{X}\colon X\to X^{\aff}$ is dominant, hence so is the composition
	$\alpha_{X}\circ \pi\colon Y\to X\to X^{\aff}$, which is equal to
	$\pi^{\aff}\circ \alpha_{Y}\colon Y\to Y^{\aff}\to X^{\aff}$.
	We conclude that $\pi^{\aff}\colon Y^{\aff}\to X^{\aff}$ is dominant,
	so it is enough to prove that $\pi^{\aff}$ is finite.
	Moreover, since $\pi^{\aff}$ is affine, 
	it is enough to prove that $\pi^{\aff}$ is proper \citestacks{04NZ}.
	
	We claim that it is enough to find a morphism of surfaces $\tau\colon Z\to Y$
	such that the compositions $\pi\circ \tau\colon Z\to X$ and 
	$(\pi\circ \tau)^{\aff}\colon Z^{\aff}\to X^{\aff}$ are finite and surjective.
	To see this, first note that $\tau$ is finite \citestacks{081Z},
	hence $\tau^{\aff}$ is dominant by the observation above.
	Since $\pi^{\aff}\circ \tau^{\aff}=(\pi\circ \tau)^{\aff}$ is proper,
	we get that $\pi^{\aff}$ is proper \citestacks{0AGD}, as needed.
	
	As in the proof of Proposition \ref{prop:covers}, we can assume that
	$\pi$ is either totally inseparable, or separable. 
	By the above reduction,
	we can assume that $\pi$ is either a power of Frobenius, or a Galois cover.
	In the first case, $\pi^{\aff}$ is a power of Frobenius, 
	too, so it is proper, as needed.
	
	Assume that $\pi$ is a Galois cover,
	i.e.\ a quotient of $Y$ by the Galois group $G$,
	see Lemma \ref{lem:lifting-Galois}\ref{item:G-quotient}.
	Since affinisation is a natural transformation,
	we get a $G$-action on $Y^{\aff}$
	such that $\alpha_{Y}$ is $G$-equivariant.
	Let $\psi\colon Y^{\aff}\to Y^{\aff}\git G$ be the quotient morphism,
	and consider the quotient $\ol{\alpha} \colon X=Y\git G\to Y^{\aff}\git G$ 
	of $\alpha_{Y}\colon Y\to Y^{\aff}$.
	Since $Y^{\aff}\git G$ is affine, 
	$\ol{\alpha}$ factors through the affinisation of $X$,
	i.e.\ $\ol{\alpha}=\varphi\circ \alpha_{X}$ for some 
	$\varphi\colon X^{\aff}\to Y^{\aff}\git G$. 
	We claim that the diagram
	\begin{equation*}
		\begin{tikzcd}
			Y \ar[dd,"\pi"] \ar[rr,"\alpha_{Y}"] && 
			Y^{\aff} \ar[ld,"\psi"] \ar[dd,"\pi^{\aff}"] \\
			& Y^{\aff}\git G &  \\
			X=Y\git G \ar[rr,"\alpha_{X}"] \ar[ru,, "\ol{\alpha}"] && 
			X^{\aff}. \ar[lu,"\varphi"]
		\end{tikzcd}
	\end{equation*}
	commutes, i.e.\ $\psi=\varphi\circ \pi^{\aff}$.
	By definition of $\ol{\alpha}$ we have
	$\psi\circ\alpha_{Y} = \ol{\alpha}\circ \pi =
	\varphi\circ\alpha_{X} \circ \pi =
	\varphi\circ \pi^{\aff}\circ\alpha_{Y}$,
	so the morphisms $\psi$ and 
	$\varphi\circ \pi^{\aff}$
	agree on $\alpha_{Y}(Y)$.
 	Since $\alpha_{Y}(Y)$ is dense in $Y^{\aff}$
	by Lemma \ref{lem:surjective-1} 
	and $Y^{\aff}\git G$ is separated, 
	we get $\psi=\varphi\circ \pi^{\aff}$, as claimed.
	Now since the quotient morphism 
	$\psi\colon Y^{\aff}\to Y^{\aff}\git G$ is finite,
	we get that $\pi^{\aff}$ is proper by \citestacks{04NX}.
\end{proof}

\begin{REM}
	In Remark \ref{rem:Iitaka} we interpret $\dim X^{\aff}$
	as the Iitaka dimension of the boundary of a compactification of $X$.
	This way, the equality $\dim Y^{\aff}=\dim X^{\aff}$ 
	in Proposition \ref{prop:proper-surj-aff} 
	can be seen as a consequence of Iitaka's 
	covering theorem \cite[Theorem 10.5]{Iitaka}.
\end{REM}	

\begin{COR}\label{cor:stable-class}
	The class of saturated surfaces with fixed dimension 
	of affinisation is stable under proper surjective morphisms.
\end{COR}
\begin{proof}
	This follows from 
	Theorem \ref{thm:functoriality}\ref{item:functoriality-iff}
	and Proposition \ref{prop:proper-surj-aff}.
\end{proof}

We end this preliminary section by the proof of the 
\enquote{if} part of Theorem \ref{thm:intro_sat_implies_proper}.
The analogous statement for scheme-saturation
was proved in \cite[Theorem 2.12]{BodBon5}.

\begin{PROP}
\label{prop:proper-implies-sat}
	Let $X$ be a surface such that
	the affinisation morphism $X\to X^{\aff}$ is proper.
    	Then $X$ is saturated.
\end{PROP}
\begin{proof}
	Since $\eta_{X}\colon X\into X^{\sat}$
	is a big open embedding, 
	we have $\alpha_{X}=\alpha_{X^{\sat}}\circ \eta_{X}$.
	By assumption, $\alpha_{X}$ is proper,
	hence so is $\eta_{X}$ \citestacks{04NX}.
	Thus, $\eta_{X}$ is an isomorphism.
\end{proof}

\subsection{Resolution of surface singularities and the intersection form}
\label{ssec:intersection form}

In order to reduce our setting to complements
of divisors in smooth projective surfaces, 
we use the following standard consequence of 
Chow's lemma and resolution of surface singularities.

\begin{PROP}[Resolution of singularities]
\label{prop:resolution}
	Let $X$ be a surface.
	The following hold.
	\begin{enumerate}
		\item\label{item:resolution-exists} 
        		There is a smooth projective surface $\ol{Y}$, 
			a simple normal crossings divisor $D$,
			and a proper birational morphism 
			$\pi\colon Y\de \ol{Y}\setminus D\to X$.
		\item\label{item:resolution-aff-sat}
        For any proper birational morphism $\pi\colon Y\to X$,
        the affinisations of $X$ and $Y$ coincide,
        i.e.\ $\alpha_Y = \alpha_X\circ\pi$.
        Moreover, $Y$ is saturated (or scheme-saturated) if and only if $X$ is.
	\end{enumerate}
\end{PROP}

\begin{proof}
	Part \ref{item:resolution-aff-sat}
	follows from Proposition \ref{prop:proper-surj-aff}
	and Theorem \ref{thm:functoriality}\ref{item:functoriality-iff}.
	We now prove part \ref{item:resolution-exists}.

    	By the Nagata compactification theorem,
    	$X$ is an open subspace of a proper surface $\ol{X}$. 
	Applying Chow's lemma \cite[Theorem IV.3.1]{Knu},
    	we find a proper birational morphism 
	$\pi_0 \colon \ol{Z}\to \ol{X}$
	such that $\ol{Z}$ is quasi-projective.
	As $\ol{Z}$ is proper, we conclude that $\ol{Z}$ is projective.
	In particular, it is schematic.
	Composing $\pi_0$ with a resolution of singularities 
	\cite{Lipman_resolution} we get a proper birational morphism 
	$\pi\colon \ol{Y}\to \ol{X}$ such that $\ol{Y}$
 	is smooth projective and 
	$D\de \ol{Y}\setminus \pi^{-1}(X)$
	is a simple normal crossings divisor.
	The restriction of 
	$\pi$ to $Y\de \ol{Y}\setminus D$
	is the required morphism. 
\end{proof}

\begin{REM}[Intersection form]
\label{rem:intersection-form}
	Proposition \ref{prop:resolution}\ref{item:resolution-exists}
	allows to define a $\mathbb{Q}$-valued intersection product 
	for proper Weil divisors on a normal surface $X$ as in \cite{Mum1}.
	Namely, let $\pi\colon Y\to X$ be as in 
	\ref{prop:resolution}\ref{item:resolution-exists},
	and let $E\subseteq Y\subseteq \ol{Y}$ be its exceptional divisor,
	with components $E_{1},\dots,E_{n}$.
	On $\ol{Y}$ we have the standard intersection product,
	and the intersection matrix $[E_{i}\cdot E_{j}]$ of $E$
	is negative definite \cite[p.\ 230]{Mum1}.
	Hence, we can define a pullback of a proper Weil divisor 
	$D\subseteq X$ as the $\mathbb{Q}$-divisor 
	$\pi^{*}D=\pi^{-1}_{*}D+\sum_{i=1}^{n}a_{i}E_i$,
	where $\pi^{-1}_{*}D$ is the proper transform of $D$,
	and $a_{1},\dots,a_{n}$ are the unique rational numbers
	such that $\pi^{*}D\cdot E_{i}=0$ for all $i$.
	For a proper divisor $D'$ on $X$ we put $D\cdot D'=\pi^{*}D \cdot \pi^{*}D'$.
    	One readily checks, see \cite{Mum1},
    	that this definition is compatible with the projection formula
    	and, in particular, does not depend on~$\pi$. 
 \end{REM} 
	
	We say that a proper Weil divisor on a normal surface is \emph{negative}
    	(\emph{semi}-)\emph{definite} if its intersection matrix is so. 
    	We now summarize some useful properties of the intersection form.
	Their proofs are easily reduced,
	using Proposition \ref{prop:resolution},
	to the case of smooth surfaces,
	which is classical, see e.g.\ 
	\cite[Examples 7.1.16, 8.3.11]{Fulton_intersection-theory}. 

\begin{LEM}
\label{lem:intersection-product}
	Let $D$ be a proper Weil divisor on a surface $X$. Then the following hold.
	\begin{enumerate}
		\item\label{item:D-pullback-def}
       			Let $\pi\colon Y\to X$ be a proper birational morphism.
            		Then $D$ is negative (semi-)definite if and only if $\pi^{*}D$ is so.
            		Moreover, if $D$ is effective, then so is $\pi^{*}D$ 
            		and the support of $\pi^{*}D$ equals 
			the preimage of the support of $D$.
		\litem{Artin's criterion}\label{item:D-Artin} 
			$D$ is contractible (Definition \ref{def:contractible})
			if and only if it is negative definite.
		\litem{Hodge Index Theorem}\label{item:D-HIT}
            		Assume that $X$ is proper,
            		and let $\operatorname{NS}_{\mathbb{Q}}(X)$ 
			be the $\mathbb{Q}$-vector space of 
			$\mathbb{Q}$-Weil divisors modulo numerical equivalence. 
            		Then $(\operatorname{NS}_{\mathbb{Q}}(X),\cdot)$ 
			has signature $(1,\rho-1)$. 
            		Moreover, if $D$ is effective,
			then $D$ is not numerically trivial.
		\end{enumerate}
	\end{LEM}	
	\begin{proof} \ref{item:D-pullback-def}
        By Proposition \ref{prop:resolution}\ref{item:resolution-exists}
        we can assume that $Y$ is smooth.
        Clearly, if $\pi^{*}D$ is negative (semi-)definite,
	then so is $D=\pi_{*}\pi^{*}D$.
        Assume that $D$ is negative semi-definite
        and let $G$ be a Weil divisor with support
	contained in $(\pi^{*}D)\redd$.
        We have $G=P+E$, where $P\de \pi^{*}\pi_{*}G$
        is a pullback of a divisor with support in $D$,
	and $E$ has support in the exceptional divisor of $\pi$.
        Then $P\cdot E=0$ and $E^2\leq 0$,
	with strict inequality if $E\neq 0$.
        It follows that $G^2\leq P^2\leq 0$,
        so $\pi^{*}D$ is negative semi-definite.
        Moreover, if $D$ is negative definite and $G\neq 0$,
        then either $P^2<0$, so $G^2<0$,
	or $P=0$ and $E\neq 0$, so $G^2=E^2<0$, as needed.
		
	Finally, assume that $D$ is effective,
	and write $\pi^{*}D=\pi^{-1}_{*}D+E$,
	where $\pi^{-1}_{*}D$ is the proper transform of $D$,
	and $E=\sum_{i=1}^{n}a_{i}E_{i}$, 
	where $E_1,\dots,E_n$ are all exceptional curves
	in the preimage of $D$.
	We need to show that $a_{i}>0$ for all $i$.
	We have $0\leq \pi^{-1}_{*}D\cdot E_i=(\pi^{*}D-E)\cdot E_i=-E\cdot E_i$
	and the inequality is strict for at least one $E_i$ 
	in each connected component of $\bigcup_{i}E_i$. 
        Hence by \cite[Lemma 3.41]{KolMor}, we have $a_{i}>0$ for all $i$, as needed.

	\ref{item:D-Artin}
	We argue as in \cite[Theorem 1.2]{Sakai_Weil-divisors}.
	By Proposition~\ref{prop:resolution}\ref{item:resolution-exists},
	there exists a proper birational morphism $\pi\colon Y\to X$
	such that $Y$ is smooth.
	By \ref{item:D-pullback-def}, 
	$D$ is negative definite if and only if 
	so is its preimage on $Y$, call it $D'$.
        By Artin's contractibility criterion 
	\cite[Corollary 6.12(b)]{Art5} and \cite[Theorem 2.3(a)]{Art1},
        $D'$ is negative definite if and only if it is contractible.
	Lemma \ref{lem:contractions-are-universal} 
	below shows that the contraction of $D'$ 
	factors as $\pi$ followed by the contraction of $D$, as needed.
		
	\ref{item:D-HIT}
	Like before, let $\pi\colon Y\to X$
	be a proper birational morphism such that 
	$Y$ is smooth projective.
	Let $\mathcal{E}$ be the subspace of 
	$\operatorname{NS}_{\mathbb{Q}}(Y)$ 
	spanned by the exceptional curves of $\pi$.
	Then $(\mathcal{E},\cdot)$ is negative definite.
	The pullback $D\mapsto \pi^{*}D$ gives an embedding
	$(\operatorname{NS}_{\mathbb{Q}}(X),\cdot)
	\into  (\operatorname{NS}_{\mathbb{Q}}(Y),\cdot)$,
	such that $\operatorname{NS}_{\mathbb{Q}}(X)^{\perp}=\mathcal{E}$. 
        The Hodge Index Theorem, see \cite[Theorem 8.3]{Iitaka},
	asserts that the intersection form on $Y$
        has signature $(1,\rho(Y)-1)$,
        which gives the required signature 
        $\operatorname{NS}_{\mathbb{Q}}(X)$.
        For the last statement, note that 
	if $D$ is an effective divisor on $X$,
        then $\pi^{*}D$ is effective by \ref{item:D-pullback-def},
        so for an ample divisor $H$ on $Y$ 
	we have $0<\pi^{*}D\cdot H=D\cdot \pi_{*}H$, as needed.
\end{proof}
	
In the above proof we used the following known result 
on uniqueness of contractions,
cf.\ \citestacks{0C5J} or \cite[Remark 1.26]{KolMor}. 

\begin{LEM}
\label{lem:contractions-are-universal}
	Let $D$ be a proper divisor on a surface $Y$.
	Let $\pi\colon Y\to X$ be a proper birational morphism
	of surfaces which maps $D$ to a point and
	is an isomorphism on $Y\setminus D$.
	Let $\varphi\colon Y\to X'$ be any proper morphism
	of separated algebraic spaces of finite type over $\kk$,
	such that $\varphi(D)$ is a point.
	Then there is a unique proper morphism 
	$\sigma\colon X\to X'$ such that $\varphi=\sigma\circ \pi$.
\end{LEM}
\begin{proof}	
	Since $\pi\colon Y\to X$ and $\varphi\colon Y\to X'$ are proper,
	the image $Z$ of $(\pi,\varphi)\colon Y\to X\times X'$ is closed
	and the projections $p\colon Z\to X$
	and $p'\colon Z\to X'$ are proper \citestacks{0AGD}.
	Since $Y$ is integral, so is $Z$.
	For $x\in X$ we have $p^{-1}(x)=\{x\}\times \varphi(\pi^{-1}(x))$, which is a point,
	hence $p$ is one-to-one.
	Since $\pi=p\circ (\pi,\varphi)\colon Y\to Z\to X$
	is birational and $Z$ is integral,
	$p\colon X\to Z$ is birational, too. 
	Thus $p\colon Z\to X$ is a proper, birational,
	one-to-one morphism of normal algebraic spaces,
	 hence an isomorphism by Zariski's Main Theorem.
	Now $p'\circ p^{-1}\colon X\to X'$
	is the required morphism $\sigma$.
	It is unique as $X'$ is separated and 
	$\sigma=\varphi\circ \pi^{-1}$ on $\pi(Y\setminus D)$.
\end{proof}

\subsection{Saturated surfaces are proper over their affinisations}
\label{ssec:prop_of_aff}	

\begin{THM}[Theorem \ref{thm:intro_sat_implies_proper}]
\label{thm:proper-surjective}
	Let $X$ be a normal surface with $\dim X^{\aff}>0$. 
    	Then $X$ is saturated if and only if the affinisation morphism 
    	$X\to X^{\aff}$ is proper. 
\end{THM}

\begin{REM}[Necessity of the assumptions in Theorem \ref{thm:proper-surjective}]
\label{rem:properness-assumptions}\ 
	\begin{parlist}
		\item\label{item:ex-non-proper}
			Examples \ref{ex:ruled}, \ref{ex:Serre}, and \ref{ex:Hironaka}\ref{item:Hironaka-0} 
			show that the assumption $\dim X^{\aff}>0$ is necessary, 
			i.e.\ there exist non-proper saturated surfaces with trivial affinisation.
		\item\label{item:ex-scheme} 
			Examples \ref{ex:scheme-sat-2}, \ref{ex:scheme-sat-1}, and \ref{ex:scheme-sat-0} 
			show that if $X$ is only scheme-saturated, 
			then the affinisation morphism $X\to X^{\aff}$
			can be non-proper regardless of the dimension of $X^{\aff}$.
	\end{parlist}
\end{REM}

\begin{proof}[Proof of Theorem \ref{thm:proper-surjective}]
	If $X\to X^{\aff}$ is proper,
    	then $X$ is saturated by Proposition \ref{prop:proper-implies-sat}.
    	Assume that $X$ is saturated.
    	By Proposition \ref{prop:resolution}
    	we can assume that $X$ is an open subspace
	of a smooth projective surface $\ol{X}$,
    	and the complement $D\de \ol{X}\setminus X$ is a divisor.
    	Embed $X^{\aff}$ as a closed subspace of some affine space $\bA^{N}$,
    	let $\ol{Y}$ be its closure in $\bP^N$,
    	and let $B$ be the scheme-theoretic intersection of $\ol{Y}$
    	with the hyperplane at infinity,
    	so $B$ is an ample divisor with support equal to $\ol{Y}\setminus X^{\aff}$.
    	Blowing up over $D$ we can assume that the affinisation morphism
    	$\alpha\colon X\to X^{\aff}$ extends to a proper surjective morphism 
    	$\ol{\alpha}\colon \ol{X}\to \ol{Y}$.
	Replacing $\ol{Y}$ by its integral closure in $\kappa(\ol{X})$,
	we obtain a proper surjective morphism 
	$\ol{\alpha}\colon \ol{X}\to \ol{Y}$ with connected fibres.
	Note that $\ol{\alpha}$ still restricts to $\alpha$ over $X$
	since $X^{\aff}$ is normal in $\kappa(\ol{X})$.

	Consider now the ample divisor $B$ supported on $\ol{Y}\setminus X^{\aff}$ 
	and its pullback $A$ to $\ol{X}$ which
	is a divisor with support contained in $D$. 
	
	Assume $\dim X^{\aff}=2$. Since $B$ is ample, we have $0<B^2=A^2$,
    	and the support of $B$ is connected \cite[Corollary III.7.9]{Har},
    	hence the same holds for $A$.
    	We claim that $D$ is connected.
    	Suppose the contrary, so $D$ has a connected component $D_0$ which is disjoint
    	from the support of $A$.
    	By the Hodge Index Theorem,
	$D_0$ is negative definite,
	hence contractible by Artin's criterion,
	see Lemma \ref{lem:intersection-product}\ref{item:D-Artin}, \ref{item:D-HIT}.
    	Since $X$ is saturated, we get a contradiction 
    	with Proposition~\ref{prop:saturation-criterion}.

    	Therefore, $D$ is connected.
	The implication \enquote{(b)$\implies$(a)} of \cite[Proposition 3.2]{Schroer}
    	shows that the affinisation morphism $X\to X^{\aff}$ is proper, as needed. 
	
	Assume $\dim X^{\aff}=1$. Then $\ol{\alpha}\colon \ol{X}\to\ol{Y}$ 
   	is a morphism onto a smooth curve, and $A$ is a sum of its fibres.
    	In particular, $A^2=0$.
    	Suppose $D$ has an irreducible component $H$ such that 
    	$\ol{\alpha}|_{H}\colon H\to \ol{Y}$ is surjective.
    	Then $H\cdot A>0$, so we can choose an integer $m\gg0$ such that $H\cdot (H+mA)>0$.
    	We have $A\cdot(H+mA)=A\cdot H>0$, so $(H+mA)^2>0$.
    	Thus $H+mA$ is nef and big, see \cite[Theorem 2.2.16]{Laz}.
    	This means that the rational map $\varphi$ induced by some multiple of $|H+mA|$
    	is birational onto its image.
    	Since $X$ is disjoint from the support of $H+mA$,
    	the map $\varphi$ restricts to a birational morphism $X\to Y$ to an affine surface $Y$.
    	By the universal property of the affinisation, $\varphi$ factors through
    	$X^{\aff}$, a contradiction as $\dim X^{\aff}=1$.
    
    	Thus, $D$ is contained in the fibres of $\ol{\alpha}$.
	If a connected component of $D$ is properly contained in a fibre,
    	then it is negative definite by Zariski's lemma,
    	see Lemma \ref{lem:fibre-type}\ref{item:F-proper-subset-is-negative}:
	as before, this is a contradiction with 
	Proposition \ref{prop:saturation-criterion} and Artin's criterion.
    	Hence, every fibre of $\ol{\alpha}$ which meets $X$ is contained in $X$.
    	Thus, we have $X=\ol{\alpha}^{-1}(\alpha(X))=\ol{\alpha}^{-1}(X^{\aff})$
    	since $\alpha(X)=X^{\aff}$ by Lemma \ref{lem:surjective-1}.
    	We conclude that $\alpha\colon X\to X^{\aff}$ is a base change of the proper morphism 
    	$\ol{\alpha}\colon \ol{X}\to \ol{Y}$ along the inclusion 
    	$X^{\aff}\into \ol{Y}$, hence $\alpha$ is proper, as claimed.
\end{proof}

\subsection{Surjectivity  of the affinisation morphism and further remarks}
\label{ssec:surj_of_aff}

Theorem \ref{thm:proper-surjective} and Lemma \ref{lem:surjective-1}
imply that if $X$ is saturated,
then the affinisation morphism $\alpha_{X}\colon X\to X^{\aff}$ is surjective.
It turns out that this weaker property  holds also for scheme-saturated surfaces. 

\begin{PROP} \label{prop:surjective-2}
	Let $X$ be a scheme-saturated surface.
	Then $X\to X^{\aff}$ is surjective.
\end{PROP}
\begin{proof} 
	If $\dim X^{\aff}\leq 1$, then $X\to X^{\aff}$ is surjective by Lemma \ref{lem:surjective-1}. 
	Assume $\dim X^{\aff}=2$.
	Since $X$ is scheme-saturated, Proposition \ref{prop:sat_and_sch_commute}
    	implies that the image of $X$ in its saturation $X^{\sat}$
    	coincides with the schematic locus of $X^{\sat}$.
    	Because $X$ is a big open subspace of $X^{\sat}$,
    	the affinisations of $X$ and $X^{\sat}$ agree, i.e. we have a commutative diagram
	\begin{equation*}
		\begin{tikzcd}
			X=(X^{\sat})^{\sch} \ar[rr,hook] \ar[rd] && X^{\sat} \ar[ld] \\ 
			& X^{\aff}=(X^{\sat})^{\aff}, & 
		\end{tikzcd}
	\end{equation*}
	where the slant arrows are the affinisation morphisms.
	By Theorem \ref{thm:proper-surjective},
	the affinisation morphism $\alpha\colon X^{\sat}\to X^{\aff}$ is proper.
	Since $\alpha_{*}\cO_{X^{\sat}}= \cO_{X^{\aff}}$, it is birational.
	Hence, any fibre $F$ of $\alpha$ is either a point in the schematic locus $X\subseteq X^{\sat}$,
	or a divisor, which meets the big open subspace $X$.
	In any case, $F$ meets $X$, so the point $\alpha(F)$ lies in $\alpha(X)$, as needed. 
\end{proof}
\begin{REM}[The affinisation is the restriction of the Iitaka fibration]\label{rem:Iitaka}
	Assume that $X$ is the complement of a reduced divisor $D$
	on a smooth projective surface $\ol{X}$ 
	(recall that we can always pass to this setting via Proposition \ref{prop:resolution}).
	Then \cite[Proposition 10.1]{Iitaka} shows that $\dim X^{\aff}$ is equal to the Iitaka dimension $\kappa(\ol{X},D)$,
	that is, the maximum of the dimensions of the images of the maps $\varphi_{m}$ induced by the linear systems $|mD|$, for all $m>0$. 
	
	We claim that the affinisation morphism $\alpha\colon X\to X^{\aff}$
	is the restriction of the Iitaka fibration $\varphi_{\infty}\colon \ol{X}\map \bP^N$ for $D$,
	which by definition is a common birational model of the maps $\varphi_m$ for all $m\gg 0$,
	see \cite[Theorems 8.6, 10.3]{Iitaka} or \cite[Theorem 2.1.33]{Laz}. 
	
	Without loss of generality,
	we can assume that $D$ has no negative definite connected components,
	i.e.\ $X$ is saturated, see Proposition \ref{prop:saturation-criterion}.
	Indeed, let $C$ be a negative definite connected component of $D$.
	Then $mC$ lies in the fixed part of $|mD|$, see e.g.\ \cite[Proposition 2.3.21]{Laz},
	hence the Iitaka fibrations for $D$ and $D-C$ coincide.
	In turn, since $X$ is a big open subspace of the algebraic space obtained from 
	$X'\de \ol{X}\setminus (D-C)$ by contracting $C$,
	we see that the affinisations of $X$ and $X'$ coincide, as needed.
	
	If $\dim X^{\aff}=0$, then both $\alpha$ and $\varphi_{\infty}$ are constant.
	Assume $\dim X^{\aff}>0$. The Iitaka fibration $\varphi_{\infty}$
	is regular on $X$ and maps it to a complement of a hyperplane section, which is affine.
	Hence its restriction to $X$ factors through $\alpha\colon X\to X^{\aff}$.
	On the other hand, blowing up over $D$ we can assume that $\alpha$
	extends to a morphism $\ol{\alpha}$ from $\ol{X}$,
	and since $\alpha$ is proper by Theorem \ref{thm:proper-surjective},
	for each $x\in X^{\aff}$ the fibre $\alpha^{-1}(x)=\ol{\alpha}^{-1}(x)$ is disjoint from $D$.
	Hence, $\kappa(D,\ol{\alpha}^{-1}(x)|_{D})=0$,
	so by \cite[Remark 2.1.35]{Laz},
	$\ol{\alpha}$ factors through the Iitaka fibration, as claimed.
	
	The above discussion shows that the Iitaka dimension $\kappa(\ol{X},D)$
	equals the transcendence degree of the fraction field of the ring $H^{0}(X,\cO_{X})$.
	By \cite[Proposition 10.1]{Iitaka}, this equality holds in higher dimension, too;
	even though the ring $H^{0}(X,\cO_{X})$ may not be finitely generated.
\end{REM}
	
\section{Saturated open subspaces}\label{sec:sat_open_sub}
	
\subsection{Numerical properties of boundaries of saturated surfaces}\label{ssec:num_boundary}	

We now describe the boundaries of compactifications 
of saturated surfaces as outlined in the introduction.
Our result is summarized in Proposition \ref{prop:numerics} below.
To state it in a convenient way,
we pose Definition \ref{def:fibre-type} which abstracts numerical properties of a fibre of a fibration.
In Lemmas \ref{lem:fibre-type} and \ref{lem:FF'}
we list its basic consequences:
they are well known, cf.\ e.g.\ \cite[Lemmas 4, 8]{Sakai-D-dimension},
but we gather the proofs for the readers' convenience.
A \emph{fibration} is a morphism with connected fibres from a surface onto a smooth curve. 
We say that a divisor $F$ is \emph{supported on} a reduced divisor $D$,
or that $D$ \emph{supports} $F$, if $F\redd=D$, i.e. the support of $F$ equals $D$.

\begin{DEF}\label{def:fibre-type}
	We say that a non-zero, proper, reduced Weil divisor on a surface is 
    	\emph{of fibre type} if it is connected and negative semi-definite,
    	but not negative definite (see Section \ref{ssec:intersection form}). 
    	We say that a divisor of fibre type is a \emph{false fibre}
    	if it does not support a fibre of a fibration.
\end{DEF}

Zariski's lemma \cite[Corollary 2.6]{Badescu} 
implies that a fibre of a proper fibration is supported on a divisor of fibre type.
In particular, it has the following numerical properties.

\begin{LEM}\label{lem:fibre-type}
	Let $D$ be a divisor of fibre type on a surface $X$. Then the following hold.
	\begin{enumerate} 
		\item\label{item:F-exists} 
			There is a non-zero divisor $F$ with support contained in $D$ such that $F^2=0$.
		\item\label{item:Fred=D} 
			We have $F\redd=D$ (i.e.\ $F$ is supported on the entire $D$), 
			$F\cdot C=0$ for every irreducible component $C$ of $D$,
			and replacing $F$ by $-F$ if needed, we can assume that $F$ is effective.
		\item\label{item:F-proper-subset-is-negative}
			Every non-zero divisor with support properly contained in $D$ is negative definite.
		\item\label{item:F-unique}
			The divisor $F$ as in \ref{item:F-exists} is unique up to taking multiples.	
		\item\label{item:F-pullback}
			Let $\pi\colon Y\to X$ be a proper birational morphism of surfaces.
			Then the preimage of $D$ on $Y$ is a divisor of fibre type.
			It is a false fibre if and only if $D$ is a false fibre.
        \item\label{item:F-extends}
			Let $X\into \ol{X}$ be an open embedding of surfaces.
			If $D$ is a fibre of a fibration $f\colon X\to B$,
			then $f$ extends to a fibration $\ol{f}\colon \ol{X}\to \ol{B}$
			for some open embedding $B\into \ol{B}$.
			In particular, a false fibre on $\ol{X}$ 
			which is contained in $X$ is also a false fibre on $X$.
	\end{enumerate} 	
\end{LEM}
\begin{proof}
	\ref{item:F-exists} This follows from directly from Definition \ref{def:fibre-type}. 
	
	\ref{item:Fred=D} Choose one divisor $F$ as in \ref{item:F-exists}.
    	Write $F=F_{+}-F_{-}$, where $F_{+}$,
    	$F_{-}$ are effective and have no common components.
    	Replacing $F$ by $-F$ if needed, we can assume $F_{+}\neq 0$.
    	We have $F_{+}\cdot F_{-}\geq 0$ and $F_{\pm}^2\leq 0$ since $D$
    	is negative semi-definite.
    	Hence $0=F^2=F_{+}^2-2F_{+}\cdot F_{-}+F_{-}^2\leq 0$, so $F_{+}^2=0$.
    	Suppose $D$ has an irreducible component $C$ such that $F_{+}\cdot C\neq 0$.
    	Then we have $(mF_{+}+C)^2=mF_{+}\cdot C+C^2>0$ for $m\gg 0$ or $m\ll 0$,
    	contrary to the fact that $D$ is negative semi-definite.
    	Since $D$ is connected, we get 
    	$(F_{+})\redd=D$ and $F_{+}\cdot C=0$ for all $C$, as needed.
    
    	\ref{item:F-proper-subset-is-negative}
    	Suppose $Z\subsetneq D$ is not negative definite.
    	Since $D$ is negative semi-definite, so is $Z$, hence 
    	there is a divisor $F_{Z}\neq 0$ with $(F_Z)\redd\subseteq Z$, $F_Z^2=0$.
	Part \ref{item:Fred=D} gives $(F_Z)\redd=D$, a contradiction. 
	
	\ref{item:F-unique} 
	Let $F,F'$ be divisors as in \ref{item:F-exists};
	by \ref{item:Fred=D} we can assume that they are effective.
	If $F$ is not a multiple of $F'$, 
	then there are integers $a,b>0$ such that the divisor $aF-bF'$ is non-zero 
    	and has support strictly contained in $D$.
    	Part \ref{item:F-proper-subset-is-negative} implies that 
	$0>(aF-bF')^2=-2abF\cdot F'$,
    	so $F\cdot F'>0$.
	But then $(F+F')^2=2F\cdot F'>0$,
    	a contradiction as $D$ is negative semi-definite.
	
	\ref{item:F-pullback} Let $D'$ be the preimage of $D$ on $Y$.
	By Lemma \ref{lem:intersection-product}\ref{item:D-pullback-def},
	$D'$ is the support of $\pi^{*} D$,
	which is negative semi-definite but not negative definite.
	Hence $D'$ is of fibre type. 
	If $D$ supports a fibre of a fibration $g\colon X\to B$,
	then $D'$ supports a fibre of the fibration $g\circ \pi\colon Y\to B$.
	Conversely, assume that $D'$ supports a fibre $F$ of a fibration $f\colon Y\to B$.
	We claim that the exceptional divisor of $\pi$ is contained in the fibres of $f$.
	Suppose the contrary, so for some exceptional curve $H$, the restriction $f|_{H}\colon H\to B$ is dominant.
	Since $H$ is proper, $f|_{H}\colon H\to B$ is surjective, so $H$ meets the fibre $F$.
	It follows that the point $\pi(H)$ lies on $\pi(F\redd)=D$,
	so $H\subseteq D'=F\redd$, a contradiction.
	The above claim and Lemma \ref{lem:contractions-are-universal}
	show that $f$ descends to a fibration of $X$ 
	with one of the fibres supported on $D$, as needed.
    
    	\ref{item:F-extends} Replacing $\ol{X}$ by its normal compactification,
	we can assume that $\ol{X}$ is proper.
	Let $\ol{B}$ be a normal compactification of $B$.
	By the Nagata compactification theorem,
	the composition of $f\colon X\to B$ with the open embedding $B\into \ol{B}$
	factors as an open embedding $X\into \wt{X}$ and a proper morphism $\wt{f}\colon \wt{X}\to \ol{B}$.
	Replacing $\wt{X}$ by its normalization and 
	resolving the indeterminacy of the birational map 
	$\begin{tikzcd}[cramped,column sep=small]\wt{X} & \ar[l,hook'] X\ar[r,hook] & \ol{X}\end{tikzcd}$,
	we get a proper birational morphism of surfaces 
	$\pi\colon \wt{X}\to \ol{X}$ which is the identity on $X$.
	In particular, $D$ supports a fibre of $\wt{f}$.
	Since $\pi$ is an isomorphism on a neighbourhood $X$ of $D$,
	the exceptional divisor of $\pi$ is contained in the fibres of $\wt{f}$.
	By Lemma \ref{lem:contractions-are-universal},
	$\wt{f}$ factors through the required fibration $\ol{X}\to \ol{B}$. 
\end{proof}

\begin{LEM}\label{lem:FF'}
	Let $D_1$ be a divisor of fibre type on a proper surface $\ol{X}$.
	Let $D_2$ be a reduced divisor on $\ol{X}$
	which is disjoint from $D_1$ and not negative definite.
	Then the following hold.
	\begin{enumerate}
		\item \label{item:FF'-fiber-type} 
			The divisor $D_2$ is of fibre type, too. 
		\item \label{item:FF'-num-prop} 
			For $j=1,2$ let $F_j$ be a divisor of self-intersection zero supported on $D_j$,
			see Lemma \ref{lem:fibre-type}.
			Then $F_1$ and $F_2$ are numerically proportional,
			i.e. there is a non-zero number $c\in \mathbb{Q}$ 
			such that $F_1\cdot C=c\cdot F_2\cdot C$ for every proper curve $C$ on $\ol{X}$.
		\item \label{item:FF'-false-fibre} If $D_1$ is a false fibre, then so is $D_2$.
	\end{enumerate}
\end{LEM}
\begin{proof}
	\ref{item:FF'-fiber-type}
	Suppose the contrary, so $G^2>0$ for some divisor $G$ with $G\redd\subseteq D_2$.
	Then $G\redd$ is disjoint from $D_1$, so by the Hodge Index Theorem,
	$D_1$ is negative definite, a contradiction. 
	
	\ref{item:FF'-num-prop} By Lemma \ref{lem:fibre-type}\ref{item:Fred=D}
	we can assume that $F_1,F_2$ are effective.
	The Hodge Index Theorem gives a basis $B_{1},\dots,B_{\rho}$ of
	$\operatorname{NS}_{\mathbb{Q}}(\ol{X})$ such that $B_{1}^2>0$, $B_{i}^2<0$ for $i>0$,
	and $B_{i}\cdot B_{j}=0$ for all $i\neq j$. In $\operatorname{NS}_{\mathbb{Q}}(\ol{X})$,
	write $F_{j}=\sum_{i=1}^{\rho} a_{ij}B_{i}$ 
	for some $a_{ij}\in \mathbb{Q}$, $j=1,2$. 
	If $a_{1j}=0$, then $0= F_{j}^2=\sum_{i>1} a_{ij}^2B_i^2$,
	so $a_{ij}=0$ for all $i$, which is impossible since $F_j$ is effective,
	see Lemma~\ref{lem:intersection-product}\ref{item:D-HIT}.
	Thus, replacing $F_1,F_2$ by multiples, we can assume $a_{11}=a_{12}$. 
	Since the supports of $F_1$ and $F_2$ are disjoint,
	we have $F_1\cdot F_2=0$,
	so $0 = (F_1-F_2)^2=\sum_{i>1}(a_{i1}-a_{i2})^2\cdot B_{i}^2$. 
	We conclude that $a_{i1}=a_{i2}$ for all $i$,
	so $F_1$ is numerically equivalent to $F_2$, as needed.
	
	\ref{item:FF'-false-fibre}
	Suppose $D_2$ supports a fibre of a fibration $f\colon \ol{X}\to B$.
	If $D_1$ has a component $H$ such that $f|_{H}\colon H\to B$ is dominant,
	then $f(H)=B$ since $H$ is proper, so $H$ meets $D_2$,
	which is impossible because $D_1$ and $D_2$ are disjoint.
	Thus, $D_1$ is contained in the fibres of $f$.
	Since $D_1$ is not negative definite, Lemma \ref{lem:fibre-type}\ref{item:F-proper-subset-is-negative}
	implies that $D_1$ supports a fibre of $f$, as needed.
\end{proof}

\begin{PROP}\label{prop:numerics}
	Let $X$ be a saturated surface.
    	Let $\ol{X}$ be a proper surface containing $X$ as an open subspace 
    	and let $D=\ol{X}\setminus X$.
    	Then $D$ is a divisor whose connected components are not negative definite,
    	and the following hold.
	\begin{enumerate}
		\item\label{item:numerics-2}
			 We have $\dim X^{\aff}=2$ if and only if $D$ is not negative semi-definite.
			 In this case $D$ is connected and supports an effective nef divisor $B$ which satisfies $B^2>0$. 
		\item\label{item:numerics-0-1}
			 We have $\dim X^{\aff}\leq 1$ if and only if every connected component of $D$ is of fibre type. 
		\item\label{item:numerics-1}
			 We have $\dim X^{\aff}=1$ if and only if there is a fibration $\ol{\alpha}\colon \ol{X}\to B$
			 such that $D$ supports a non-empty union of some of its fibres.
			 In this case, one can choose $\ol{\alpha}$ so that it restricts to the affinisation morphism $X\to X^{\aff}$. 
		\item\label{item:numerics-0}
			 We have $\dim X^{\aff}=0$ if and only if $D$ is a sum of disjoint false fibres.
	\end{enumerate} 	 
\end{PROP}

\begin{proof}
	The first statement is proved in Proposition \ref{prop:saturation-criterion}.
	We now prove \ref{item:numerics-2}--\ref{item:numerics-0}.
	By Proposition \ref{prop:resolution} and Lemma \ref{lem:fibre-type}\ref{item:F-pullback},
	we can assume that $\ol{X}$ is smooth projective.
	
	\ref{item:numerics-2}
	The surface $X$ is \emph{almost affine} in the sense of \cite[\S 3]{Schroer}
	if and only if $\dim X^{\aff}=2$ and the affinisation morphism $X\to X^{\aff}$ is proper.
	Since $X$ is saturated, by Theorem \ref{thm:proper-surjective}
	this holds if and only if $\dim X^{\aff}=2$.
	Thus, \ref{item:numerics-2} follows from \cite[Proposition 3.2]{Schroer}. 

	\ref{item:numerics-0-1}
	This follows from \ref{item:numerics-2},
    	as the connected components of $D$ are not negative definite.
	
	\ref{item:numerics-1}
	A fibration $\ol{\alpha}$ as in \ref{item:numerics-1}
    	maps $X$ to the affine curve $B\setminus \ol{\alpha}(D)$,
    	so $\dim X^{\aff}\geq 1$.
	Since $D$ is contained in the fibres of $\ol{\alpha}$,
	it is negative semi-definite, so $\dim X^{\aff}=1$ by \ref{item:numerics-2}. 
	
	Conversely, assume $\dim X^{\aff}=1$.
	Since $\alpha\colon X\to X^{\aff}$ is proper by Theorem \ref{thm:proper-surjective}
	and $\alpha_{*}\cO_{X}=\cO_{X^{\aff}}$,
	we find that $\alpha$ has connected fibres 
	by Zariski's Connectedness Theorem, see \citestacks{0A1B}.
	Hence $\alpha\colon X\to X^{\aff}$ is a proper fibration.
	By Lemma \ref{lem:fibre-type}\ref{item:F-extends},
	it extends to a proper fibration $\ol{\alpha}\colon \ol{X}\to B$,
	and $D$ supports the fibres over $B\setminus X^{\aff}$, as needed.

	\ref{item:numerics-0}
    	This follows from \ref{item:numerics-0-1}, \ref{item:numerics-1},
	and Definition \ref{def:fibre-type} of a false fibre.
\end{proof} 	

\begin{REM}[Proposition \ref{prop:numerics} via Iitaka dimension]\label{rem:Iitaka-numerics}
	Proposition \ref{prop:numerics} can also be seen as a consequence of 
	Proposition \ref{prop:saturation-criterion} and 
	known properties of the Iitaka dimension, as follows.
	By Proposition \ref{prop:resolution}
	we can assume that $\ol{X}$ is smooth projective.
	Then by Remark \ref{rem:Iitaka} we have $\dim X^{\aff}=\kappa(\ol{X},D)$.
	Now Proposition \ref{prop:numerics}\ref{item:numerics-2}
	follows from \cite[Lemma 8.5 and Proposition 8.6]{Iitaka};
	\ref{prop:numerics}\ref{item:numerics-1} follows from \cite[Theorem 8.6]{Iitaka},
	and \ref{prop:numerics}\ref{item:numerics-0-1},\ref{item:numerics-1} are their formal consequences.
\end{REM}

\subsection{False fibres}\label{ssec:false-fibres}

In this section, we prove Theorem \ref{thm:intro-two-false-fibres}.
In particular, we prove that any disjoint collection of false fibres has at most $2$ elements 
(Corollary \ref{cor:at-most-two}),
and additionally restrict the case with exactly two 
(Proposition \ref{prop:two-false-fibres}).
This result  improves Proposition \ref{prop:numerics}
with a criterion allowing to distinguish the case 
$\dim X^{\aff}=0$ from $\dim X^{\aff}=1$,
i.e.\ a false fibre from a genuine fibre of a fibration.
As shown by Examples \ref{ex:Hironaka} and \ref{ex:Serre},
such a criterion cannot be purely numerical. 

\smallskip

The proof uses the Albanese variety, which we now briefly recall.
For a smooth projective variety $\ol{X}$,
its \emph{Albanese variety} is an Abelian variety 
$\Alb(\ol{X})$ together with the \emph{Albanese morphism} 
$a\colon \ol{X}\to \Alb(\ol{X})$ such that the image 
$a(\ol{X})$ generates $\Alb(\ol{X})$ as a group,
and for any morphism $f\colon \ol{X}\to A$ to an abelian variety 
we have $f=h\circ a+c$ for a unique homomorphism $h\colon \Alb(\ol{X})\to A$ 
and a unique constant $c\in A$. 
For the existence and basic properties of $\Alb(\ol{X})$
see e.g.\ \cite[\S IV]{Lang_Abelian}.
By Theorem VI.1 loc.\ cit, the dual to $\Alb(\ol{X})$ is the
Picard variety $\Pic^{0}\redd(\ol{X})$ parametrising line bundles 
which are algebraically equivalent to $0$, modulo linear equivalence.
By \cite[9.6.3]{FGAexp} 
(see also \cite[19.3.1]{Fulton_intersection-theory} for the case $\kk=\mathbb{C}$),
every numerically trivial divisor on $\ol{X}$ 
has a multiple which corresponds to an algebraically trivial line bundle,
so its linear equivalence class falls in $\Pic^{0}(\ol{X})$.
If $C$ is a smooth projective curve,
then $\Pic^{0}(C)\cong\Alb(C)$ is isomorphic to the Jacobian of $C$,
in particular its dimension is the genus $g(C)$ of $C$ \cite[Theorem VI.3]{Lang_Abelian}.
If $\ol{X}$ is a smooth projective surface, then $\dim \Pic\redd^{0}(\ol{X})$ is called the
\emph{irregularity} of $\ol{X}$ and is denoted by $q(\ol{X})$.

We now prove the main technical result of this section.
For a stronger version in characteristic~$0$, see Corollary \ref{cor:kernel} below.

\begin{LEM}\label{lem:Pic-restriction}
	Let $D$ be a false fibre on a smooth projective surface $\ol{X}$.
	Assume that all irreducible components $D_{1},\dots,D_{n}$ of $D$ are smooth.
	Then the restriction homomorphism
	\begin{equation*}
		 \Pic\redd^{0}(\ol{X})\to \prod_{i=1}^{n}\Pic^{0}(D_i)
	\end{equation*}
	has finite kernel.
\end{LEM}
\begin{proof}
	The restriction homomorphism $\Pic^{0}\redd(\ol{X})\to \Pic^{0}(D_{i})$ is dual to the homomorphism
    	$\Alb(D_i)\to \Alb(\ol{X})$ induced by the inclusion $D_{i}\into \ol{X}$.
    	Therefore, by \cite[Proposition V.2]{Lang_Abelian}
    	it is enough to prove that $\prod_{i=1}^{n}\Alb(D_i)\to \Alb(\ol{X})$ is surjective.
    	Denote the image of this homomorphism by $A$,
    	and let $f\colon \ol{X}\to \Alb(\ol{X})/A$ be the composition of the Albanese morphism 
	$a\colon \ol{X}\to \Alb(\ol{X})$ with the quotient $\Alb(\ol{X})\to \Alb(\ol{X})/A$. 
	Since $a(\ol{X})$ generates $\Alb(\ol{X})$,
	its image $f(\ol{X})$ generates $\Alb(\ol{X})/A$.
    	Thus, it is enough to prove that $f(\ol{X})$ is a point. 
	
	The restriction $a|_{D_i}\colon D_{i}\into \ol{X}\to \Alb(\ol{X})$
	factors as the composition of the Albanese morphism $D_{i}\to \Alb(D_{i})$,
	the above homomorphism $\Alb(D_i)\to \Alb(\ol{X})$, and a translation.
	Therefore, $a(D_i)\subseteq \Alb(\ol{X})$ is a coset of $\Alb(D_i)$,
	so $f(D_i)$ is a point.
	Since $D$ is connected, it follows that 
	$f\colon \ol{X}\to \Alb(\ol{X})/A$ maps $D$ to a point.
		
	Because $D$ is not negative definite, 
	we have $\dim f(\ol{X})\leq 1$. 
	If $\dim f(\ol{X})=1$, 
	then the Stein factorisation of $q$ gives a fibration of $\ol{X}$ 
	with one of its fibres supported on $D$, 
	contrary to the definition of a false fibre. 
	Thus, $\dim f(\ol{X})=0$, as needed.
\end{proof}

\begin{PROP}[Two disjoint false fibres]\label{prop:two-false-fibres}
	Let $D,D'$ be disjoint false fibres on a smooth projective surface $\ol{X}$.
	Let $F$ and $F'$ be effective divisors
	of self-intersection $0$ supported on $D$ and $D'$,
	respectively, see Lemma~\ref{lem:fibre-type}.
	Let $D_{1},\dots,D_{n}$ be all irreducible components of $D$,
	let $\wt{D}_{i}$ be the normalisation of $D_i$,
	and denote its genus by $g(\wt{D}_i)$.
	Then the following hold.
	\begin{enumerate}
		\item\label{item:false-fibre-difference}
			For some integers $m,m'$,
			the divisor $mF-m'F'$ is a non-torsion element of $\Pic^{0}(\ol{X})$. 
		\item\label{item:false-fibre-normal}
			The pullback of $\cO_{\ol{X}}(F)$ to some $\wt{D}_i$ is a non-torsion line bundle of degree $0$. 
		\item\label{item:false-fibre-genus-irregularity}
			We have $0<q(\ol{X})\leq \sum_{i=1}^{n}g(\wt{D}_i)$
			and $\kk$ is not an algebraic closure of a finite field.
		\item\label{item:false-fibre-2}
			Let $D''\neq D,D'$ be another false fibre. 
        		Then $D''$ meets $D$ and $D'$.
	\end{enumerate}
\end{PROP}	
\begin{proof}
	Let $\pi\colon \ol{Y}\to \ol{X}$ be a proper birational morphism of smooth projective surfaces
such that all irreducible components of $\pi^{-1}(D)$ are smooth.
    	By Lemma \ref{lem:fibre-type}\ref{item:F-pullback}
	the preimages of false fibres on $\ol{X}$ are false fibres on $\ol{Y}$. 
	The inequality in \ref{item:false-fibre-genus-irregularity}
	is the same for $D$ and for $\pi^{-1}(D)$: 
	indeed, we have $q(\ol{Y})=q(\ol{X})$, 
    	the irreducible components of the proper transform of $D$
	are isomorphic to $\wt{D}_1,\dots,\wt{D}_n$,
	and the exceptional curves have genus~$0$. 
	Thus, replacing $D,D',D''$ by their preimages on $\ol{Y}$, and $F,F'$ by their pullbacks,
    	we can assume that $D_1,\dots,D_n$ are smooth.
	
	\ref{item:false-fibre-difference} 
    	By Lemma \ref{lem:FF'}\ref{item:FF'-num-prop},
    	we can replace $F,F'$ by some 
    	non-zero multiples so that $F-F'$ is numerically trivial. 
    	By \cite[9.6.3]{FGAexp}, replacing $F-F'$ by its multiple 
    	we can assume that it is algebraically trivial,
	i.e.\ $F-F'\in \Pic^{0}(\ol{X})$.
	Suppose it is torsion, so once again 
	replacing it by a multiple we get a linear equivalence $F-F'\sim 0$.
	So there is a rational function whose divisors
	of zeros and poles are equal to $F$ and $F'$, respectively.
	Since the supports of $F$ and $F'$ are disjoint,
	this function extends to a morphism $\ol{X}\to \mathbb{P}^1$
	such that $F$ and $F'$ are fibres over $0$ and $\infty$.
	Its Stein factorization gives a fibration 
	with fibres supported on $D$ and $D'$, a contradiction.
	
	\ref{item:false-fibre-normal}
	By \ref{item:false-fibre-difference},
	we can replace $F$ and $F'$ by some multiples so that $F-F'$ 
    	is a non-torsion element of $\Pic^0(\ol{X})$. 
    	By Lemma \ref{lem:Pic-restriction}, its restriction 
    	to $\prod_{i=1}^{n}\Pic^0(D_i)$ is non-torsion, too, 
    	so $(F-F')|_{D_i}$ for some $i$ is a non-torsion element of $\Pic^{0}(D_i)$. 
    	Since $F'$ is supported away from $D_i$, 
    	we get that $F|_{D_i}$ is a non-torsion element of $\Pic^{0}(D_i)$, as needed.
	
	\ref{item:false-fibre-genus-irregularity} 
	Part \ref{item:false-fibre-difference}
	implies that $\Pic^{0}(\ol{X})$ is not a torsion group,
	so $q(\ol{X})=\dim \Pic^0(\ol{X})>0$ and $\kk$
	is not an algebraic closure of a finite field.
	The inequality $q(\ol{X})\leq \sum_{i}g(D_i)$ holds by Lemma~\ref{lem:Pic-restriction}. 
	
	\ref{item:false-fibre-2}
    	Suppose the contrary, so, say, $D''$ is disjoint from $D'$.
	Suppose $D''$ meets $D$.
	Then $(D+D'')\redd$ is a connected reduced divisor 
	which is disjoint from $D'$ and not negative definite,
	hence of fibre type by Lemma \ref{lem:FF'}\ref{item:FF'-fiber-type}.
	Since $D$ and $D''$ are of fibre type, too,
	by Lemma \ref{lem:fibre-type}\ref{item:F-proper-subset-is-negative}
	we get $D=(D+D'')\redd=D''$, a contradiction.
    
    	Thus $D,D',$ and $D''$ are pairwise disjoint.
    	By \ref{item:false-fibre-difference}, we have divisors $F'',F'$ 
    	supported on $D'',D'$ such that $F''-F'$ is non-torsion in $\Pic^{0}(\ol{X})$. 
    	By Lemma \ref{lem:Pic-restriction},
    	the restriction of $F''-F'$ to some $D_i$ is non-zero. 
    	This is a contradiction, as $F''-F'$ is supported away from $D$.
\end{proof}

\begin{COR}\label{cor:at-most-two}
	Let $D_1,D_2,D_3$ be false fibres on a surface $\ol{X}$.
	Then for some $i\in \{1,2,3\}$ we have $D_{i}\cap D_j\neq \emptyset$ for both $j\neq i$.
\end{COR}
\begin{proof}
	By Proposition \ref{prop:resolution}\ref{item:resolution-exists}
	and Lemma \ref{lem:fibre-type}\ref{item:F-pullback},\ref{item:F-extends},
	we can compactify $\ol{X}$ and resolve its singularities,
	and thus assume that $\ol{X}$ is smooth projective. 
	If, say, $D_1\cap D_2=\emptyset$,
	then by Proposition \ref{prop:two-false-fibres}\ref{item:false-fibre-2}
	we have $D_3\cap D_i\neq \emptyset$ for $i=1,2$, as needed.
\end{proof} 

\begin{COR}[Theorem \ref{thm:intro-two-false-fibres}]
\label{cor:two-false-fibres}
	Let $X$ be a saturated surface such that $X^{\aff}$ is a point.
	Let $\ol{X}$ be a proper surface containing $X$ as an open subspace,
	and let $D\de \ol{X}\setminus X$.
	Then $D$ has at most two connected components.
\end{COR}
\begin{proof}
	By Proposition \ref{prop:numerics}\ref{item:numerics-0}
	the connected components of $D$ are false fibres.
	The result follows from Corollary \ref{cor:at-most-two}.
\end{proof}

\begin{REM}[Normal bundle to a false fibre]
\label{rem:normal-bundle}
	Let $C$ be a smooth curve on a smooth projective surface $\ol{X}$,
	and let $\cN_{C/\ol{X}}\de \cO_{\ol{X}}(C)|_{C}$ be its normal bundle,
	so $\deg \cN_{C/\ol{X}}=C^2$.
	Assume that $C$ is of fibre type,
	i.e.\ $\deg \cN_{C/\ol{X}}=0$.
	We claim that if $\cN_{C/\ol{X}}$ is non-torsion,
	then $C$ is a false fibre.
	
	Suppose the contrary,
	so $mC$ for some $m>0$ is a fibre of a fibration $f\colon \ol{X}\to B$,
	say $mC=f^{*}b$ for some $b\in B$.
	Then $\cN_{C/\ol{X}}^{\otimes m}=\cO_{\ol{X}}(mC)|_{C}=f^{*}\cO_{B}(b)|_{b}=\cO_{C}$, a contradiction.
	
	Proposition \ref{prop:two-false-fibres}\ref{item:false-fibre-normal}
	shows that the converse is true if $\ol{X}\setminus C$ contains a false fibre.
	In general, the converse is false:
	Example \ref{ex:Serre} shows that the normal bundle to a false fibre can be trivial.
\end{REM}

\begin{REM}[Mumford's criterion in positive characteristic]
\label{rem:Mumford}
	As in Remark \ref{rem:normal-bundle}, 
	let $C$ be a smooth curve on a smooth projective surface $\ol{X}$ such that $C^2=0$.
	Assume $\cha \kk>0$.
	Then \cite[Proposition on p.\ 336]{Mumford-Enriques-1} asserts that $C$ is a false fibre
	if and only if for every integer $n\geq 1$
	we have $\cO_{\ol{X}}(nC)/\cO_{\ol{X}}\not \cong \cO_{\ol{X}}/\cO_{\ol{X}}(-nC)$.
	In particular, taking $n=1$ we get that in positive characteristic
	the normal bundle to a false fibre must be non-trivial.
\end{REM}
\begin{proof} 
	Since the proof of loc.\ cit.\ is left to the reader,
	we spell it out for completeness. 

	The curve $C$ is of fibre type,
	so either $C$ is a false fibre and $h^0(\ol{X},\mathcal{O}_{\ol{X}}(nC))=1$ for all $n\geq 1$,
	or $C$ supports a fibre of a fibration and $h^0(\ol{X},\mathcal{O}_{\ol{X}}(nC))\xrightarrow{n\to \infty}\infty$.
	Put $\mathcal{F}_n:=\mathcal{O}_{\ol{X}}(nC)/\mathcal{O}_{\ol{X}}$.
	From the long exact cohomology sequence of
	$0\to \mathcal{O}_{\ol{X}} \to \mathcal{O}_{\ol{X}}(nC)\to\mathcal{F}_n\to 0$
	we observe that $h^0(\ol{X},\mathcal{O}_{\ol{X}}(nC))\xrightarrow{n\to \infty}\infty$
	if and only if $h^0(\ol{X},\mathcal{F}_n)\xrightarrow{n\to \infty}\infty$.
	Thus it is enough to check that $h^0(\ol{X},\mathcal{F}_n)\xrightarrow{n\to \infty}\infty$
	if and only if $\cF_{n}\cong \cO_{\ol{X}}/\cO_{\ol{X}}(-nC)$ for some $n\geq 1$.
	
	Let $i\colon C\into \ol{X}$ be the closed immersion of $C$ into $\ol{X}$.
	Let $\cL\de i^{*}\cO_{\ol{X}}(C)$ be its normal bundle.
	We have a canonical exact sequence $0\to \cO_{\ol{X}}\to \cO_{\ol{X}}(C)\to i_{*}\cL\to 0$,
	which after tensoring with $\cO_{\ol{X}}((n-1)C)$ gives 
	$0\to \cO_{\ol{X}}((n-1)C)\to \cO_{\ol{X}}(nC)\to i_{*}\cL^{\otimes n}\to 0$,
	as $i_{*}\cL\otimes \cO_{\ol{X}}((n-1)C)=
	i_{*}(\cL\otimes i^{*}\cO_{\ol{X}}((n-1)C))=
	i_{*}(\cL\otimes \cL^{\otimes (n-1)})=
	i_{*}\cL^{\otimes n}$.
	Thus we obtain $i_{*}\cL^{\otimes n}\cong \cO_{\ol{X}}(nC)/\cO_{\ol{X}}((n-1)C)\cong \cF_{n}/\cF_{n-1}$.
	We get an exact sequence
	\begin{equation*}
		0\to \cF_{n-1}\to \cF_{n}\to i_{*}\cL^{\otimes n}\to 0,
	\end{equation*} 	
	where $\cF_{n-1}\to \cF_{n}$ is the multiplication by the canonical section 
	$\cO_{\ol{X}}\to \cO_{\ol{X}}(C)$,
	which restricts to zero on $C$,
	so the map $\cF_{n}\to i_{*}\cL^{\otimes n}$
	restricts to an isomorphism $i^{*}\cF_{n}\cong \cL^{\otimes n}$.

	We claim that $h^0(\ol{X},\mathcal{F}_n)>h^0(\ol{X},\mathcal{F}_{n-1})$ if and only if 
	$\mathcal{F}_n\cong \mathcal{O}_{\ol{X}}/\mathcal{O}_{\ol{X}}(-nC)$.

	Assume $\mathcal{F}_n\cong \mathcal{O}_{\ol{X}}/\mathcal{O}_{\ol{X}}(-nC)$.
	We have a surjective morphism 
	$\mathcal{O}_{\ol{X}}\twoheadrightarrow \mathcal{F}_n\twoheadrightarrow i_{*}\mathcal{L}^{\otimes n}$,
	so we have a global section of $\mathcal{F}_n$
	which maps to a non-zero global section of $i_{*}\mathcal{L}^{\otimes n}$.
	The above exact sequence shows that $h^0(\ol{X},\mathcal{F}_n)>h^0(\ol{X},\mathcal{F}_{n-1})$.

	Conversely, assume that $h^0(\ol{X},\mathcal{F}_n)>h^0(\ol{X},\mathcal{F}_{n-1})$.
	Then some global section $s\in H^0(\ol{X},\mathcal{F}_n)$ maps to a non-zero section 
	$i^{*}s\in H^{0}(\ol{X},\cL^{\otimes n})$.
	Since $\deg_{C}\mathcal{L}^{\otimes n}=(nC)^2=0$, 
	we find that $i^{*}s$ induces an isomorphism 
	$\mathcal{O}_C\xrightarrow{\sim}\mathcal{L}^{\otimes n}\cong i^{*}\cF_{n}$.
	It follows that $s\colon \mathcal{O}_{\ol{X}}\to \mathcal{F}_n$ is surjective.
	Indeed, on $\ol{X}\setminus C$ the sheaf $\mathcal{F}_n$ vanishes, 
	whereas $i^{\ast}s$ is an isomorphism on $C$, 
	so the support of the cokernel of $s$ does not contain closed points, i.e is empty.

	Consider the kernel $\mathcal{K}$ of $s$.
	Since the projective dimension of 
	$\mathcal{F}_n=\cO_{\ol{X}}(nC)/\cO_{\ol{X}}(C)$ is $1$ (on each affine chart of $\ol{X}$)
	and $\mathcal{K}$ admits a resolution by vector bundles (as $\ol{X}$ is smooth projective), 
	we see that $\mathcal{K}$ is locally free,
	\cite[\href{https://stacks.math.columbia.edu/tag/0CXC}{Lemma 0CXC}]{stacks-project}.
	Further, we can compute the determinant of $\mathcal{F}_n$ 
	via the short exact sequences
	$0\to \mathcal{O}_{\ol{X}}\to \mathcal{O}_{\ol{X}}(nC)\to \mathcal{F}_n\to 0$
	or $0 \to \mathcal{K}\to \mathcal{O}_{\ol{X}}\to \mathcal{F}_n\to 0$
	to get $\mathcal{K}\cong \mathcal{O}_{\ol{X}}(-nC)$.
	Hence $\mathcal{F}_n\cong \mathcal{O}_{\ol{X}}/\mathcal{O}_{\ol{X}}(-nC)$, as claimed.

	Recall that $\cha\kk=p>0$.
	Pulling back line bundles along the absolute Frobenius of~$\ol{X}$
	raises them to the $p$-th powers,
	so pulling back the isomorphism 
	$\cF_{n}\cong \mathcal{O}_{\ol{X}}/\mathcal{O}_{\ol{X}}(-nC)$
	gives isomorphisms 
	$\cF_{np^k}\cong \mathcal{O}_{\ol{X}}/\mathcal{O}_{\ol{X}}(-np^kC)$ 
	for all $k>0$.
	Thus, $h^0(\ol{X},\mathcal{F}_n)> h^0(\ol{X},\mathcal{F}_{n-1})$
	happening at least once is equivalent to it happening for infinitely many $n$,
	and we get the desired equivalence.

	For the last statement of the remark, note that if the normal bundle $\cL$ to $C$ is trivial,
	then $\cF_{1}\cong i_{*}\cL\cong i_{*}\cO_{C}\cong \cO_{C}/\cO(-C)$,
	so $C$ is not a false fibre.
\end{proof}

\begin{REM}[False fibres of low genus]\label{rem:Sakai} 
	Let $D$ be a divisor of fibre type on a smooth projective surface $\ol{X}$,
	and let $F$ be a divisor with $F^2=0$ supported on $D$.
	Assume $F\cdot K_{\ol{X}}<0$:
	this holds e.g.\ if $D\cong \bP^1$,
	by the adjunction formula.
	In this case,
	it is known that $|F|$ induces a $\bP^1$-fibration of $\ol{X}$,
	hence $D$ is not a false fibre 
	(to see this, use the asymptotic Riemann--Roch theorem for $F$,
	or apply \cite[Theorem 1.28(2)]{KolMor} to a minimal model of $\ol{X}$).
	
	Assume $F\cdot K_{\ol{X}}=0$ (e.g.\ $D$ is an elliptic curve). 
	Then either $D$ is a false fibre, 
	or $|F|$ induces an elliptic or quasi-elliptic fibration of $\ol{X}$.
	This case was studied in \cite{Sakai-D-dimension},
	under the assumption $\cha\kk=0$. Theorem 1 loc.\ cit.\ implies that 
	if $D$ is a false fibre, then either $\ol{X}$ is rational 
	(as in Example \ref{ex:Hironaka}),
	or $D$ contains a section of a $\bP^1$-fibration $\ol{X}\to C$ over an elliptic curve 
	(as in Examples \ref{ex:ruled}, \ref{ex:Serre}).
	Furthermore, if $\ol{X}$ contains two disjoint false fibres,
	then the second case holds (note that $q(\ol{X})=1$ by 
	Proposition \ref{prop:two-false-fibres}\ref{item:false-fibre-genus-irregularity}),
	and $\ol{X}$ is a blowup of $\bP(\cO_{C}\oplus \cL)$ for some non-torsion line bundle 
	$\cL$ of degree $0$ on $C$, see Example \ref{ex:ruled}\ref{item:ex-ruled-2}.
\end{REM}

Our next result is a stronger version of 
Lemma \ref{lem:Pic-restriction} in characteristic $0$.
We do not know if this result is optimal, 
i.e. if \eqref{eq:restriction} can have non-trivial kernel.

\begin{COR}\label{cor:kernel}
	Assume $\cha\kk=0$. 
	Let $D$ be a false fibre on a smooth projective surface $\ol{X}$
	such that all its irreducible components
	$D_1,\dots,D_n$ are smooth.
	Then the restriction map
	\begin{equation}
	\label{eq:restriction}
		 \Pic\redd^{0}(\ol{X})\to \prod_{i=1}^{n}\Pic^{0}(D_i)
	\end{equation}
	has kernel of order at most $2$.
\end{COR}

\begin{proof}
	Let $K$ be the kernel of \eqref{eq:restriction},
	and let $\mathcal{L}\in \Pic^{0}(\ol{X})$ be a line bundle in $K$.
	By Lemma \ref{lem:Pic-restriction}, $\cL$ is torsion,
	i.e.\ $\cL^{\otimes n}\cong \cO_{\ol{X}}$ for some $n\geq 1$.
	Since $\cha\kk=0$, the line bundle $\cL$ 
	determines a finite \'etale cover $\pi\colon \ol{X}'\to \ol{X}$ with Galois
	group $\mathbb{Z}/n\mathbb{Z}$, see e.g.\ \cite[Proposition 4.1.6]{Laz}.

	By Proposition \ref{prop:numerics}\ref{item:numerics-0},
	$X\de \ol{X}\setminus D$ is saturated with trivial affinisation.
	Let $X'=\pi^{-1}(X)$, and consider the finite \'etale cover $\pi|_{X'}\colon X'\to X$.
	By Corollary \ref{cor:stable-class}, $X'$ is saturated with trivial affinisation,
	too, so the connected components of $D'\de \ol{X}\setminus X'=\pi^{-1}(D)$ are false fibres.
	However, since $\mathcal{L}$ lies in $K$,
	we have $\cL|_{D_i}\cong \cO_{D_i}$,
 	so $D'$ has $\deg \pi=n$ connected components.
	By Corollary \ref{cor:at-most-two} we conclude that $n \leq 2$,
	that is, $\mathcal{L}$ has order at most $2$.
	
	Suppose that $K$ contains distinct non-trivial line bundles
	$\mathcal{L}_1\neq \mathcal{L}_2$.
	Consider the corresponding finite \'etale covers
	$\ol{X}'_i\to \ol{X}$, $i=1,2$.
	The fibre product $\pi:\ol{X}'_1\times_{\ol{X}}\ol{X}'_2\to X$
	is a finite \'etale cover of degree $4$.
	Since $\mathcal{L}_1$ and $\mathcal{L}_2$ lie in $K$,
	$\pi^{-1}(D)$ has $4$ connected components. 
	As before, we see that those connected components are false fibres,
	a contradiction with Corollary \ref{cor:at-most-two}.
\end{proof}

\begin{REM}[Characterization of false fibres]
\label{rem:alternative-false-fibre}
	Let $D$ be a divisor of fibre type on a proper surface $\ol{X}$.
	Corollary \ref{cor:at-most-two} implies that $D$ is a false fibre
	if and only if $\ol{X}$ contains at most one divisor of fibre type which is disjoint from $D$. 
\end{REM}

\begin{REM}[{Comparison with \cite[Theorem 3.3]{Schroer}}]
\label{rem:Schroer}	
	Let $D$ be a connected reduced divisor on a projective surface $\ol{X}$.
	If $D$ is negative definite,
	then \cite[Theorem 3.3]{Schroer} asserts that $D$ supports
	a fibre of a proper birational morphism $\ol{X}\to \ol{Y}$ 
	if and only if the complement $\ol{X}\setminus D$ contains an effective divisor $R$ such that $R^2>0$. 
	In turn, if $D$ is of fibre type,
	then Remark \ref{rem:alternative-false-fibre}
	asserts that $D$ supports a fibre of a proper fibration $\ol{X}\to \ol{Y}$
	if and only if $\ol{X}\setminus D$ contains an effective divisor $R$
	such that $R^2=0$ and $R$ has non-connected support.
	
	In the Minimal Model Program,
	morphisms $\ol{X}\to \ol{Y}$ as above are called 
	\emph{divisorial} and \emph{fibre type} contractions, respectively,
	see \cite[Proposition 2.5]{KolMor}.
	In each case, if $\ol{X}\to \ol{Y}$ exists,
	then the required $R$ can be obtained as a pullback of a very ample divisor on $\ol{Y}$.
	Thus both criteria can be informally summarized as follows:
	\emph{a contraction of $D$ exists if and only if $\ol{X}\setminus D$
	contains an effective divisor which is a candidate for a pullback of a very ample one.}
\end{REM}

\subsection{Saturated surfaces which are non-proper over their affinisations}
\label{ssec:SA-min}
We now summarize our previous results into Theorem \ref{thm:dimXaff=0},
characterising surfaces $X$ for which the answer to 
A.\ Bondal's question is negative  -- by Theorem \ref{thm:proper-surjective},
these are non-proper saturated surfaces with $\dim X^{\aff}=0$.
First, we give a characterisation of 
saturated surfaces with $\dim X^{\aff}=1$
which does not require fixing a compactification.

\begin{PROP}\label{prop:dimXaff=1}
	Let $X$ be a non-proper saturated surface.
    	Then the following are equivalent.
	\begin{enumerate}[(i)]
		\item\label{item:equiv-1}
			 We have $\dim X^{\aff}=1$.
		\item\label{item:equiv-fibration}
			 The surface $X$ admits a proper fibration onto a smooth affine curve.
		\item\label{item:equiv-proper-curves-above}
			 The surface $X$ contains infinitely many proper curves.
		\item\label{item:equiv-proper-curves-below} 
			For every proper birational morphism $X\to Y$ of surfaces,
			the surface $Y$ contains at least two proper curves.
		\item\label{item:equiv-nonneg-def}
			The surface $X$ contains at least two proper divisors 
			with non-negative self-intersection numbers and different supports.
		\item\label{item:equiv-fibre-type}
			The surface $X$ contains at least two different divisors of fibre type.
	\end{enumerate}
\end{PROP}
\begin{proof}
	Given a proper birational morphism $Y\to X$,
	all the above conditions are preserved after replacing $X$ by $Y$:
	indeed, for saturation and $\dim X^{\aff}$
	see Proposition \ref{prop:resolution}\ref{item:resolution-aff-sat};
	and for divisors of fibre type see Lemma \ref{lem:fibre-type}\ref{item:F-pullback}.
	Hence by Proposition \ref{prop:resolution}\ref{item:resolution-exists},
	we can assume that $X$ is a complement of a reduced divisor $D$
	on a smooth projective surface $\ol{X}$.
	We have $D\neq 0$ since $X$ is non-proper.
	Since $X$ is saturated, $D$ is not negative definite 
	(Proposition \ref{prop:saturation-criterion}).
	
	To see \ref{item:equiv-1}$\implies$\ref{item:equiv-fibration}, 
    	consider the affinisation morphism $\alpha\colon X\to X^{\aff}$. 
    	It is proper by Theorem \ref{thm:proper-surjective},
    	hence a fibration because $\alpha_{*}\cO_{X}=\cO_{X^{\aff}}$,
    	cf.\ Proposition \ref{prop:numerics}\ref{item:numerics-1}.
    	Clearly \ref{item:equiv-fibration} implies \ref{item:equiv-proper-curves-above}:
    	take fibres of the fibration from \ref{item:equiv-fibration}.
    	The implication \ref{item:equiv-proper-curves-above}
    	$\implies$\ref{item:equiv-proper-curves-below}
	follows from the fact that the exceptional divisor 
	of a proper birational morphism has finitely many irreducible components.
	
	To see \ref{item:equiv-proper-curves-below}$\implies$\ref{item:equiv-nonneg-def},
    	let $E$ be a maximal (in the sense of inclusion)
    	reduced divisor on $\ol{X}$ which is disjoint from $D$
    	and whose intersection matrix is negative definite (possibly $E=0$).
    	Such $E$ exists because the N\'eron-Severi group 
    	$\operatorname{NS}_{\mathbb{Q}}(\ol{X})$ has finite rank.
    	By Artin's criterion, see Lemma~\ref{lem:intersection-product}\ref{item:D-Artin}, 
    	there is a proper birational morphism $\ol{\pi}\colon \ol{X}\to\ol{Y}$
    	which contracts $E$ and is an isomorphism on $\ol{X}\setminus E$.
    	By \ref{item:equiv-proper-curves-below},
    	the image $Y=\ol{\pi}(X)$ contains two proper curves.
    	Let $C_1,C_2$ be their proper transforms on $\ol{X}$.
    	Then $E+C_i$ are two divisors with different supports contained in $X$,
    	whose intersection matrices are not negative definite by definition of $E$, as required.
    
	For \ref{item:equiv-nonneg-def}$\implies$\ref{item:equiv-fibre-type},
	denote by $D_1,D_2$ the supports of the two divisors as in \ref{item:equiv-nonneg-def}.
	For $i=1,2$ we have $D\cdot D_i=0$ and both $D$ and $D_i$ are not negative definite,
	so the Hodge Index Theorem implies that $D$ and $D_i$ are of fibre type, as needed. 
	
	For \ref{item:equiv-fibre-type}$\implies$\ref{item:equiv-1}
	note that by Lemma \ref{lem:FF'}\ref{item:FF'-fiber-type},
	$D$ is of fibre type, so $\dim X^{\aff}\leq 1$ by 
    	Proposition \ref{prop:numerics}\ref{item:numerics-0-1}.
    	Suppose $\dim X^{\aff}=0$, so $D$ is a false fibre.
    	The two divisors of fibre type in $X$ are disjoint from $D$,
    	so by Lemma \ref{lem:FF'}\ref{item:FF'-false-fibre} they are false fibres, too.
    	By Corollary \ref{cor:at-most-two} one of them meets $D$, a contradiction.  
\end{proof}

\begin{THM}\label{thm:dimXaff=0}
	Let $X$ be an open subspace of a proper surface $\ol{X}$,
	and let $D=\ol{X}\setminus X$.
	Then the following are equivalent.
	\begin{enumerate}[(i)]
		\item\label{item:equiv-0-aff-non-proper}
			$X$ is saturated, but the affinisation morphism $X\to X^{\aff}$ is not proper.
		\item\label{item:equiv-0-aff-trivial}
			$X$ is saturated, non-proper, and $X^{\aff}$ is a point.
		\item\label{item:equiv-0-false-fibres}
			$D$ is a non-empty disjoint union of divisors of fibre type 
			and there is at most one reduced divisor contained in $X$
			which is not negative definite.
		\item\label{item:equiv-0-finitely-many-curves}
			$D$ is a non-empty disjoint union of divisors of fibre type
			and $X$ contains finitely many proper curves.
	\end{enumerate}
\end{THM}
\begin{proof}
	The equivalence \ref{item:equiv-0-aff-non-proper}$\iff$\ref{item:equiv-0-aff-trivial}
	holds by Theorem \ref{thm:proper-surjective}. 
	If $X$ is not saturated, then \ref{item:equiv-0-aff-trivial} fails, 
	and by Proposition \ref{prop:saturation-criterion} $D$
	has a connected component which is a point or a negative definite divisor,
	so \ref{item:equiv-0-false-fibres}, \ref{item:equiv-0-finitely-many-curves} fail, too.
	Assume $X$ is saturated.
	If $\dim X^{\aff}=2$, 
	then \ref{item:equiv-0-aff-trivial} fails,
	and by Proposition \ref{prop:numerics}\ref{item:numerics-2} $D$
	is not negative semi-definite,
	so again \ref{item:equiv-0-false-fibres}, \ref{item:equiv-0-finitely-many-curves} fail. 
	Assume $\dim X^{\aff}\leq 1$. 
    	By Proposition \ref{prop:numerics}\ref{item:numerics-0-1},
    	$D$ is a disjoint union of divisors of fibre type.
    	Then the equivalence of \ref{item:equiv-0-aff-trivial},
	\ref{item:equiv-0-false-fibres}, and \ref{item:equiv-0-finitely-many-curves}
    	follows from the equivalence of \ref{item:equiv-1}, 
    	\ref{item:equiv-nonneg-def}, and \ref{item:equiv-proper-curves-above} of Proposition 
    	\ref{prop:dimXaff=1}.
\end{proof}

As our last result, we note another consequence of Corollary \ref{cor:at-most-two}.
It suggests that saturated surfaces which are 
non-proper over their affinisation are rather special; but nonetheless, 
they exist in any birational class, at least if the field $\kk$ is large enough. 

\begin{PROP}\label{prop:SA-minimal}
	Let $X$ be a saturated surface.
	Let $\cS(X)$ be the set of saturated open subspaces $U\subseteq X$
	with $\dim U^{\aff}=\dim X^{\aff}$, ordered by inclusion.
	The following hold.
	\begin{enumerate}
		\item\label{item:SA-min-exists}
			The set $\cS(X)$ has a minimal element if and only if $\dim X^{\aff}=0$. 
		\item\label{item:SA-birational}
			Assume that $X$ is not proper,
			$\dim X^{\aff}=0$, and $\cS(X)=\{X\}$,
			i.e.\ every saturated open subspace 
			strictly contained in $X$ has non-trivial affinisation. 
        		Then for every proper birational morphism $Y\to X$ we have $\cS(Y)=\{Y\}$, too.
		\item\label{item:SA-birational-converse}
			Conversely, if $X$ is proper (hence $\dim X^{\aff}=0$)
			and $\kk$ is not an algebraic closure of a finite field,
			then there is a proper birational morphism $Y\to X$
			such that $\cS(Y)\neq \{Y\}$, i.e.\ $Y$ contains
			a non-proper surface with trivial affinisation.
	\end{enumerate}
\end{PROP}
\begin{proof}
	\ref{item:SA-min-exists}
	Assume first that $\dim X^{\aff}>0$.
	The affinisation morphism $\alpha\colon X\to X^{\aff}$ is surjective 
	by Theorem \ref{thm:proper-surjective} and Lemma \ref{lem:surjective-1},
	so any descending chain of affine open subspaces 
	$X^{\aff}\supsetneq V_{1}\supsetneq V_{2}\supsetneq \cdots$
	lifts to a descending chain 
	$X\supsetneq \alpha^{-1}(V_1)\supsetneq \alpha^{-1}(V_2)\supsetneq \cdots$.
    	Hence, it is enough to show that for every affine open subspace $V\subseteq X^{\aff}$,
    	its preimage $U\de \alpha^{-1}(V)$ belongs to $\cS(X)$.
    	Since $V$ is affine,
    	the restriction $\alpha|_{U}\colon U\to V$ factors as a composition 
    	$\alpha|_{U}=h\circ \alpha_{U}$ of the affinisation 
	$\alpha_{U}\colon U\to U^{\aff}$ and some morphism $h\colon U^{\aff}\to V$.
    	By Theorem \ref{thm:proper-surjective}, $\alpha$ is proper,
    	hence so is its base change $\alpha|_{U}$, and consequently $\alpha_{U}$ \citestacks{04NX}.
    	By Proposition \ref{prop:proper-implies-sat}, $U$ is saturated.
    	As before, we see that $\alpha_{U}$ is surjective,
    	so $h$ is proper, too \citestacks{0AGD}.
    	But $h$ is affine, so $h$ is finite, and since
    	$h_{*}\cO_{U^{\aff}}=\alpha_{*}\cO_{U}=\cO_{V}$,
    	we conclude that $h$ is an isomorphism. 
    	In particular, $U^{\aff}=V$ has the same dimension as $X^{\aff}$, as needed.
	
	Assume now that $\dim X^{\aff}=0$. By Proposition \ref{prop:resolution} 
    	we can assume that $X$ is a complement of a reduced divisor $D$
    	in a smooth projective surface. 
    	Indeed, if $\pi\colon Y\to X$ is a proper birational morphism 
    	as in Proposition \ref{prop:resolution}, then any descending chain
    	$X\supsetneq U \supsetneq \cdots$ in $\cS(X)$ lifts to a descending 
    	chain $Y\supsetneq \pi^{-1}(U)\supsetneq \cdots$ in $\cS(Y)$.
    	Now if $X=X_0\supsetneq X_1\supsetneq \cdots$ is a descending chain in $\cS(X)$, 
    	then by Proposition \ref{prop:numerics}\ref{item:numerics-0}
    	the complement of each $X_i$ is a disjoint union of at least $i$ false fibres, 
    	hence $i\leq 2$ by Corollary \ref{cor:at-most-two}.
    	Thus, $\cS(X)$ has a minimal element, as needed.
	
	\ref{item:SA-birational}
	Let $\ol{X}$ be a proper surface containing $X$ as an open subspace.
	By Proposition \ref{prop:numerics}\ref{item:numerics-0},
	$D\de \ol{X}\setminus X$ consists of false fibres.
	Suppose $\cS(Y)\neq \{Y\}$, so $Y$ contains a false fibre.
	Let $D'$ be its image on $X$. It is a divisor on $\ol{X}$
	which is disjoint from $D$ and not negative definite,
	hence a false fibre by Lemma \ref{lem:FF'}\ref{item:FF'-false-fibre}.
	Therefore, $X\setminus D'\subsetneq X$ is an element of $\cS(X)$, a contradiction.
	
	\ref{item:SA-birational-converse}
	This is proved in \cite[Example 1]{Sakai-D-dimension}
	by adapting a classical construction of Hironaka, see Example \ref{ex:Hironaka}.
	For the readers' convenience, we recall the argument in more detail. 

	By Proposition \ref{prop:resolution},
	we can assume that $X$ is smooth projective.
	There is a smooth, non-rational curve $C\subseteq X$ 
	such that $\cO_{X}(C)$ is very ample.
	To see this, we fix a closed embedding $X\into \mathbb{P}^{N}$.
	By Bertini's theorem \cite[II.8.18]{Har} for a general
	hypersurface $C'\subseteq \mathbb{P}^N$ of degree $d$ the intersection
	$C\de C'\cap X$ is a smooth curve.
	Clearly, $\cO_{X}(C)$ is very ample.
	Let $L'$ be a hyperplane in $\mathbb{P}^{N}$,
	and let $L=L'\cap X$.
	The linear equivalence $C'\sim dL'$ on $\mathbb{P}^N$
	restricts to $C\sim dL$ on $X$.
	By the adjunction formula, $C$ has genus
	$g(C)=\frac{1}{2}C\cdot(C+K_{X})+1=\frac{1}{2}dL \cdot (dL +K_{X})+1$.
	Since $\cO_{X}(L)$ is very ample, we have $L^2>0$,
	and therefore $g(C)\gg 0$ for $d\gg 0$, as needed.
	
    	Since $\cO_{X}(C)$ is very ample, we have $n\de C^2>0$,
    	and the restriction $\cO_X(C)|_{C}$ is isomorphic to
	$\cO_C(q_{1}+\dots+q_{n})$ for some points $q_{1},\dots,q_{n}\in C$.
    	Since $C\not\cong \bP^1$, the Jacobian $\Pic^{0}(C)$ is non-trivial, 
	and since $\kk$ is not an algebraic closure of a finite field,
	$\Pic^{0}(C)$ has infinite rank \cite[Theorem 10.1]{FJ}. 
    	Thus, we can choose $n$ distinct points
	$p_{1},\dots,p_{n}\in C$ such that the divisor 
	$\sum_{i=1}^n(p_i-q_i)$ is non-torsion in $\Pic^{0}(C)$.
	Let $\pi\colon Y\to X$ be a blowup at
	$p_1,\dots,p_n$, and let $D\de \pi^{-1}_{*}C$
	be the proper transform of $C$.
	Then $D$ is a smooth curve with $D^2=0$.
    	To show that $Y\setminus D$ is the required element of $\cS(Y)$,
	we need to show that $D$ is a false fibre.
	
	Suppose the contrary, 
	so $mD$ for some $m>0$ is a fibre of a fibration $Y\to B$.
    	Let $F\neq mD$ be another fibre,
    	and let $E_{i}=\pi^{-1}(p_i)$ be the exceptional curve over $p_i$.
    	Then $F$ is disjoint from $D$ and satisfies $F\cdot E_i=mD\cdot E_i=m$. 
    	So $R\de\pi_{*}F$ meets $C$ only in the points $p_{1},\dots,p_{n}$, 
    	with multiplicity $m$, i.e.\  $R|_{C}=m\sum_{i=1}^{n}p_{i}$.
    	Therefore, the linear equivalence $R\sim mC$ restricts to 
	a linear equivalence $m\sum_{i=1}^{n}p_{i}\sim m\sum_{i=1}^{n}q_{i}$.
    	We obtain the equality $m\sum_{i=1}^{n}(p_i-q_i)=0$ in $\Pic^{0}(C)$
    	contradicting the choice of the points $p_1,\dots, p_n$.
\end{proof}

\begin{REM}[Case $\kk=\mathbb{F}_{p}^{\textnormal{alg}}$]
\label{rem:Fpalg}
	Assume that $\kk$ is an algebraic closure of a finite field.
	Then every surface is schematic, cf.\ \cite[Corollary 2.11]{Art1}.
	Indeed, let $X$ be a surface, 
	taking its normal compactification we can assume that it is proper.
	Let $\pi\colon Y\to X$ be a resolution as in 
	Proposition \ref{prop:resolution}\ref{item:resolution-exists}.
	The exceptional divisor of $\pi$ is negative definite,
	so it can be contracted in the category of schemes 
	by \cite[Proposition 2.9(B)]{Art1}, as needed.

	Therefore, in this case the notion of saturation introduced in
	Definition \ref{def:saturation} coincides with the scheme-saturation studied in \cite{BodBon5}.
	However, Examples in Section \ref{ssec:examples-scheme-sat}
	show that whenever $\kk$ is not an algebraic closure of a finite field,
	Definition \ref{def:saturation} is strictly more restrictive.
\end{REM}

\section{Examples}
\label{sec:exampl}

We now give some examples illustrating the necessity of our assumptions,
announced throughout the paper.
In this section, we assume that 
\begin{equation}\label{eq:assumption-k}\tag{$*$}
\mbox{the ground field $\kk$ is not an algebraic closure of a finite field.}
\end{equation}
By \cite[Theorem 10.1]{FJ},
assumption \eqref{eq:assumption-k} implies that non-trivial Abelian varieties,
e.g.\ Jacobians of smooth projective non-rational curves,
have infinite rank and are not torsion.

\subsection{Non-proper saturated surfaces with trivial affinisations}
\label{ssec:examples-dimXaff-0}
We start with examples of surfaces for which the answer to A.\ Bondal's question is negative.
By Theorem \ref{thm:intro_sat_implies_proper},
these are precisely non-proper surfaces with trivial affinisation,
or equivalently, by Proposition \ref{prop:numerics}\ref{item:numerics-0},
complements of disjoint unions of false fibres on proper surfaces.

\begin{EXM}[False fibre as a section of a ruled surface]
\label{ex:ruled}
	Let $C$ be a smooth, projective, non-rational curve.
	Let $\cL$ be a non-torsion line bundle of degree $0$ on $C$;
	it exists by assumption~\eqref{eq:assumption-k}.
	Let $\ol{X}=\bP(\cO_{C}\oplus \cL)$,
	and let $D$ be the section of $\ol{X}\to C$
	corresponding to the projection $\cO_{C}\oplus \cL\to \cL$,
	see \cite[Proposition 2.6]{Har}.
	\begin{parlist}
		\item\label{item:ex-ruled-1}
			Let $X=\ol{X}\setminus D$.
			The normal bundle $\cN_{D/\ol{X}}$ to $D$
			can be identified with $\cL$.
			In particular, $D^2=\deg\cL=0$,
			so by Proposition \ref{prop:saturation-criterion}
			the surface $X$ is saturated,
			and by Proposition \ref{prop:numerics}\ref{item:numerics-0-1}
			we have $\dim X^{\aff}\leq 1$.
			Since $\cN_{D/\ol{X}}$ is non-torsion,
			$D$ cannot support a fibre of a fibration 
			(Remark \ref{rem:normal-bundle}),
			 hence $\dim X^{\aff}=0$ by Proposition \ref{prop:numerics}\ref{item:numerics-0}.
		\item \label{item:ex-ruled-2}
			Let $D_1=D$, and let $D_2$
			be the section corresponding to the other projection 
			$\cO_{C}\oplus \cL\to \cO_{C}$. 
			As in \ref{item:ex-ruled-1}, we see that $D_1,D_2$ are disjoint,
			smooth projective curves with self-intersection zero
			and non-torsion normal bundles,
			hence each surface $X_{i}=\ol{X}\setminus D_i$,
			as well as their intersection $X=\ol{X}\setminus (D_1\sqcup D_2)$,
			is a non-proper surface with trivial affinisation.
			 Moreover, the surface $X$ shows that number $2$ in 
			Theorem \ref{thm:intro-two-false-fibres} is optimal. 
	\end{parlist}	
\end{EXM}

\begin{EXM}[False fibre with trivial normal bundle, $\cha\kk=0$]
\label{ex:Serre}
	In this classical example, due to Serre,
	we get $\dim X^{\aff}=0$ even though the normal bundle to the boundary $D$ is trivial.
	Note that by Remark \ref{rem:Mumford},
	this is impossible if $\cha \kk>0$.
	We follow \cite[Example VI.3.2]{Har_ample}. 
	
	Assume $\cha \kk=0$.
        Let $C$ be an elliptic curve.
        Since $\Ext^{1}(\cO_{C},\cO_{C})=H^{1}(C,\cO_{C})=\kk$,
        there is a non-trivial vector bundle $\cE$ of rank $2$ on $C$
        which is an extension of $\cO_{C}$ by $\cO_{C}$.
        Let $\ol{X}$ be the associated ruled surface $\bP(\cE)$,
        and let $D\subseteq \ol{X}$ be the section 
        corresponding to the surjection $\cE\to \cO_{C}$.
        The normal bundle to $D$ is trivial, so $D^2=0$.
	As in Example \ref{ex:ruled}, 
	using Propositions \ref{prop:saturation-criterion} and 
	\ref{prop:numerics}\ref{item:numerics-0-1}
	we conclude that $X:=\ol{X}\setminus D$ 
        is saturated and $\dim X^{\aff}\leq 1$. 
        
        Suppose $\dim X^{\aff}=1$.
	Then by Proposition \ref{prop:numerics}\ref{item:numerics-1},
	the divisor $pD$ for some $p\geq 1$ is a fibre 
	of some fibration $\ol{\alpha}$ of $\ol{X}$.
	Let $F$ be any fibre of $\ol{\alpha}$.
	Then $pF$ is a $p$-section of the ruling $\mathbb{P}(\cE)\to C$, 
	i.e. the ruling restricted to $F$ is a finite morphism of degree $p$.
	Such $p$-sections correspond to a subbundles of degree $0$ of $\operatorname{Sym}^{p}\cE$,
	see \cite[Proposition I.10.2]{Har_ample}.
	Atiyah's classification \cite{Atiyah-elliptic}
	implies that there is only one such subbundle,
	corresponding to $pD$, a contradiction. 

	Note that if $\cha\kk=p>0$, 
	then this construction yields a surface $X$ with $\dim X^{\aff}=1$.
	Indeed, by \cite[Lemma 2.8]{TU_elliptic} we have $h^{0}(C,\operatorname{Sym}^{p}\cE)=2$,
	so we get a one-parameter family of $p$-sections as above.
	Hence $D$ is not a false fibre,
	see Theorem \ref{thm:dimXaff=0}.
	In fact, it follows from \cite[Lemmas 2.12(i), 2.14]{TU_elliptic}
	that this family is an elliptic fibration of $\ol{X}$. 
		
	The fact that $\dim X^{\aff}$ depends on the characteristic of $\kk$
	suggests that no purely numerical criterion can distinguish between cases 
	$\dim X^{\aff}=0$ and $1$, cf.\ Example \ref{ex:Hironaka}.
		
	Finally, we note that if $\kk=\bC$,
	then $X$ is biholomorphic to 
	$\bC^{*}\times \bC^{*}$ \cite[p.\ 234]{Har_ample},
	so it has plenty of non-constant holomorphic functions,
	even though it has no regular one.
\end{EXM}

\subsection{Numerical properties cannot detect false fibres} 
\label{ssec:examples-numerics}
Example \ref{ex:Hironaka} below shows that the strict transform
of a planar cubic blown up in $9$ points can be a false fibre or not,
depending on the position of those points.
Therefore, no purely numerical property can tell if a divisor of fibre type is a false fibre or not,
or equivalently, if a saturated surface $X$ has $\dim X^{\aff}=0$ or $1$.

Example \ref{ex:Hironaka} and the subsequent ones are variants 
of the classical construction of Hironaka \cite[Example V.7.3]{Har}.
For the readers' convenience we recall its basic properties.

\begin{LEM}[{Contractibility of the proper transform of a planar cubic, cf.\ \cite[2.23]{Palka_almost_MMP}}]
\label{lem:Hironaka}
	Let $\ol{C}\subseteq \bP^2$ be a smooth cubic.
	Let $\sigma\colon \ol{X}\to \bP^2$ be a blowup at 
	distinct points $p_{1},\dots, p_{n}\in \ol{C}$, $n\geq 10$,
	and let $C\subseteq \ol{X}$ be the proper transform of $\ol{C}$.
	Then $C^2=9-n<0$, so by Artin's criterion there exists 
	a proper birational morphism $\ol{X}\to \ol{Y}$ of surfaces which contracts $C$
	and is an isomorphism on $X\setminus C$.
	The following hold.
	\begin{enumerate}
		\item\label{item:contractible}
			The surface $\ol{Y}$ is projective (equivalently, schematic)
			if and only if $\{p_{1},\dots,p_{n}\}=\ol{C}\cap \ol{R}$
			for some effective divisor $\ol{R}\subseteq \bP^2$
			such that the local intersection multiplicity 
			$(\ol{C}\cdot \ol{R})_{p_i}$ equals the multiplicity of 
			$\ol{R}$ at $p_i$ for all $i=1,\dots,n$.
		\item \label{item:non-contractible} 
			Assume that the points $p_{1},\dots,p_{n}$ are 
			$\bZ$-linearly independent in the group law of $\ol{C}$;
			such points exist by assumption \eqref{eq:assumption-k}.
			Then $\ol{Y}$ is not schematic.
	\end{enumerate}
\end{LEM}
\begin{proof}
	\ref{item:contractible}
	Assume such $\ol{R}$ exists.
	Clearly, $d\de \deg \ol{R}\geq \frac{n}{3}>3$.
	We have $\ol{R}\sim \ol{C}+(d-3)\ol{L}$,
	where $\ol{L}$ is any line not passing through $p_{1},\dots,p_{n}$.
	Let $R,L$ be the proper transforms of $\ol{R}$, $\ol{L}$ on $\ol{X}$.
	Then $R\sim C+(d-3)L$.
	Since $R$ is disjoint from $C$,
	and we can choose lines $\ol{L}$, $\ol{L}'$
	as above such that $L\cap L'$ is disjoint from $R$,
	we see that the linear system $|R|$ is base point free.
	Since $R^2>0$, for $m\gg 0$ the linear system $|mR|$
	induces a birational morphism of projective surfaces,
	whose exceptional divisor is the only curve 
	which intersects $R$ trivially, namely $C$.
	
	Conversely, assume $\ol{Y}$ is schematic.
	Since $\ol{Y}$ is proper and has only one singular point,
	it is projective by \cite[p.\ 328, Corollary 4]{Kleiman}.
	The pullback to $\ol{X}$ of any hyperplane section of $\ol{Y}$
	which does not pass through $\operatorname{Sing}\ol{Y}$ is disjoint from $C$,
	so its image on $\mathbb{P}^2$ is the required $\ol{R}$.
	
	\ref{item:non-contractible}
	Suppose $\ol{Y}$ is schematic.
	Then by \ref{item:contractible}, 
	there is an effective divisor $\ol{R}$ meeting 
	$\ol{C}$ in exactly $p_{1},\dots, p_{n}$.
    	Put $d=\deg \ol{R}$, $m_{i}=(\ol{R}\cdot \ol{C})_{p_i}$,
    	so $3d=\sum_{i=1}^{n}m_i$.
    	Let $p_0\in \ol{C}$ be the inflection point of $\ol{C}$ which we take as $0$
    	in the group law, and let $\ell$ be the line tangent to $\ol{C}$ at $p_0$.
    	The linear equivalence $\ol{R}\sim d\ell$ on $\mathbb{P}^2$
	restricts to a linear equivalence 
	$\sum_{i=1}^{n}m_{i}p_{i}\sim 3dp_0$ on $\ol{C}$,
	which gives $\sum_{i=1}^{n}m_{i}(p_i-p_0)\sim 0$,
	hence $\sum_{i=1}^{n}m_{i}p_{i}=0$
	in the group law on $\ol{C}$, a contradiction.
\end{proof}

\begin{EXM}
\label{ex:Hironaka} 
	Let $\ol{C}\subseteq \bP^2$ be a smooth cubic.
	Let $\sigma\colon \ol{X}\to \bP^2$ be 
	a blowup at $9$ distinct points $p_{1},\dots, p_{9}\in \ol{C}$,
	let $C\subseteq \ol{X}$ be the strict transform of $\ol{C}$
	and let $X=\ol{X}\setminus C$. 
	We have $C^2=0$, 
	so by Proposition \ref{prop:numerics}\ref{item:numerics-0-1}
	the surface $X$ is saturated and $\dim X^{\aff}\leq 1$.
		
	\begin{parlist}
		\item \label{item:Hironaka-1}
			Assume that $p_{1},\dots,p_{9}$
			are the common points of $\ol{C}$
			and another cubic $\ol{R}$.
			Then the pencil of cubics spanned by 
			$\ol{C}$ and $\ol{R}$ lifts to a fibration 
			$\ol{X}\to \bP^1$ with $C$ being one of its fibres,
			so $\dim X^{\aff}=1$ by 
			Proposition \ref{prop:numerics}\ref{item:numerics-1}.
		\item \label{item:Hironaka-0}
		 	Assume that $p_{1}+\dots+p_{9}$ is non-torsion in the group law on $\ol{C}$,
			this is possible by assumption \eqref{eq:assumption-k}.
			Suppose $\dim X^{\aff}=1$,
			so $mC$ for some $m>0$ is a fibre of a fibration of $\ol{X}$,
			see Proposition \ref{prop:numerics}\ref{item:numerics-1}.
			Let $F$ be another fibre.
			The image $\ol{F}\de \sigma(F)$ meets $\ol{C}$ only in
			$p_{1},\dots, p_{9}$, with multiplicity $m$.
			We have linear equivalence $\ol{F}\sim d\ell$,
			where $d=\deg \ol{F}=3m$, and $\ell$ is the line tangent to $C$
			at the inflection point $p_0$ which we take as the neutral element in the group law.
			Restricting this linear equivalence to $\ol{C}$,
			we get $\sum_{i=1}^{9}mp_{i}\sim m p_{0}$,
			so $m\cdot \sum_{i=1}^{9}p_{i}=0$ in the group law on $\ol{C}$,
			contrary to our assumption.
			We conclude that $\dim X^{\aff}=0$.
	\end{parlist}
\end{EXM}

\subsection{Scheme-saturated surfaces which are non-proper over their affinisations}
\label{ssec:examples-scheme-sat}
Finally, we give some examples illustrating the failure of 
Theorem \ref{thm:intro_sat_implies_proper} in the category of schemes,
announced in Remark \ref{rem:properness-assumptions}\ref{item:ex-scheme}.
That is, we construct surfaces $X$ with $\dim X^{\aff}=2,1,0$ which are scheme-saturated,
but not saturated, hence non-proper over $X^{\aff}$ (Proposition \ref{prop:proper-implies-sat}).

In the case when $\dim X^{\aff}=2$,
treated in Example \ref{ex:scheme-sat-2},
we use the construction from \cite[Example 3.26]{Palka_almost_MMP}
of two log exceptional curves of the second kind 
whose contractions in the course of the Minimal Model Program
must be performed in a specific order, see \S 2C loc.\ cit.\ for details.
The case $\dim X^{\aff}=1$ (Example \ref{ex:scheme-sat-1}),
is a slight modification of the same construction.
An example with $\dim X^{\aff}=0$ was given in \cite[Example 2.13]{BodBon5},
we recall it in Example \ref{ex:scheme-sat-0}.

\begin{EXM}[$\dim X^{\aff}=2$]
\label{ex:scheme-sat-2} 
	Let $\ol{C}\subseteq \bP^2$ be a smooth cubic,
	let $\ol{Y}\to \bP^2$ be a blowup at $p_{1},\dots, p_{n}\in \ol{C}$, $n\geq 10$,
	and let $\wt{C}$ be the proper transform of $\ol{C}$.
	By Lemma \ref{lem:Hironaka}\ref{item:contractible}
	we can choose $p_{1},\dots,p_{n}$ in such a way that there is 
	a birational morphism of projective surfaces $\pi\colon \ol{Y}\to \ol{Z}$
	whose exceptional divisor equals $\wt{C}$.
	Let $\ol{A}$ be a hyperplane section of $\ol{Z}$ which does not pass 
	through the singularity $q\de \pi(\wt{C})$ of $\ol{Z}$,
	so $U\de \ol{Z}\setminus \ol{A}$ is an affine neighbourhood of $q$.
		
	Let $Y=\pi^{-1}(U)=\ol{Y}\setminus \pi^{-1}\ol{A}$.
	Let $\tau\colon \ol{X}\to \ol{Y}$ be a blowup at a point $p\in \wt{C}$
	such that the proper transform $C$ of $\wt{C}$ is \emph{not} schematically contractible,
	it exists by Lemma \ref{lem:Hironaka}\ref{item:non-contractible}.
	Let $E$ be the exceptional divisor of $\tau$, let $A=(\pi\circ\tau)^{-1}\ol{A}$, $D=C+A$,
	and $X=\ol{X}\setminus D$.
	We have $A^2=\ol{A}^2>0$ and $A\cong \ol{A}$,
	which is connected by \cite[Corollary III.7.9]{Har}.
	Hence the connected components of $D$ are:
	$A$ with $A^2>0$, and $C$ which is negative definite,
	but not schematically contractible.
	Thus, $X$ is scheme-saturated by Proposition \ref{prop:saturation-criterion},
	see \cite[Theorem 2.14]{BodBon5}.
		
	We claim that $\alpha\de (\pi\circ\tau)|_{X}\colon X\to U$
	is the affinisation morphism.
	Since $\alpha$ is surjective and $U$ is affine,
	it is enough to show that 
	$(\pi\circ\tau)^{*}\colon H^{0}(U,\cO_{U})\to H^{0}(X,\cO_{X})$ 
	is an isomorphism.
	Because $C^2<0$ is contractible 
	(in the category of algebraic spaces),
	Hartogs's lemma gives $H^{0}(\cO_{X},X)=H^{0}(\cO_{X'},X')$,
	where $X'=X\cup C=\ol{X}\setminus A$.
	Since $\pi\circ \tau\colon X'\to U$ is proper and birational,
	we have $(\pi\circ \tau)_{*}\cO_{X'}=\cO_{U}$, which proves the claim.

  	We conclude that $X$ is scheme-saturated, $\dim X^{\aff}=2$,
	but $X$ is not proper over $X^{\aff}$, because the fibre of 
	$\alpha\colon X\to X^{\aff}$ over the singular point $q$ is $E\cap X \cong \bA^1$.
	It follows from  Theorem \ref{thm:proper-surjective} that $X$ is not saturated.
	In fact, Proposition \ref{prop:saturation-criterion} implies that 
	$X^{\sat}=\ol{W}\setminus \varphi(A)$, 
	where $\varphi\colon \ol{X}\to \ol{W}$ is the contraction of $C$. 
	This is illustrated by the commutative diagrams: 
		\begin{equation*}
			\begin{tikzcd}[column sep=6em]
				&  
				\ol{X} \ar[d, "\mbox{\tiny{blow up }} p"', "\tau"] \ar[r,"\mbox{\tiny{contract }} C", "\varphi"'] & 
				\ol{W} \ar[d] &   \hspace{-6em}  \leftarrow \mbox{ \small{alg.\ space}}
				\\
				\bP^2 & 
				\ar[l, "\mbox{\tiny{blow up }} {p_1, \dots, p_n}", "\sigma"'] 
				\ol{Y} \ar[r,"\mbox{\tiny{contract }} \wt{C}", "\pi"'] & 
				\ol{Z} & \hspace{-7.5em}  \leftarrow \mbox{ \small{scheme}} 
			\end{tikzcd} 
			\quad 
			\mbox{restricts to:} \quad
			\begin{tikzcd}
				X \ar[d] \ar[r,] & 
				X^{\sat} \ar[d] 
				\\
				Y \ar[r] & 
				X^{\aff}. 
			\end{tikzcd} 
		\end{equation*}
\end{EXM}

\begin{EXM}[$\dim X^{\aff}=1$]
\label{ex:scheme-sat-1}
	Let $\ol{C}\subseteq \bP^2$ be a smooth cubic,
	and let $\ol{R}\subseteq \bP^2$ be another cubic meeting 
	$\ol{C}$ normally in $p_{1},\dots, p_{9}$.
	Let $\sigma\colon \ol{X}\to \bP^2$ be a blowup at 
	$p_{1},\dots,p_{9},p_{10}\in \ol{C}$,
	where the point $p_{10}$ is chosen in such a way that 
	the proper transform $C$ of $\ol{C}$ is not schematically contractible,
	see Lemma \ref{lem:Hironaka}\ref{item:non-contractible}. 
	Let $R$ be the proper transform of $\ol{R}$,
	$D=C+R$, $E=\sigma^{-1}(p_{10})$, and $X=\ol{X}\setminus D$.
	The connected components of $D$ are:
	$R$ with $R^2=0$, and $C$,
        which is contractible, but not schematically contractible.
	By Proposition \ref{prop:saturation-criterion},
	$X$ is scheme-saturated, but not saturated,
        and $X^{\sat}$ is obtained from $\ol{X}\setminus R$ by contracting $C$.

	The pencil of cubics spanned by $\ol{C}$ and $\ol{R}$ induces a fibration
	$\ol{\alpha}\colon \ol{X}\to \bP^1$ such that $C+E$ and $R$ are fibres of
	$\ol{\alpha}$, say $C+E=\ol{\alpha}^{-1}(0)$ and $R=\ol{\alpha}^{-1}(\infty)$,
	so $\ol{\alpha}$ restricts to $\alpha\colon X\to \bA^1$. 		
		
	As in Example \ref{ex:scheme-sat-2}, we show that $\alpha$ is the affinisation morphism.
	Indeed, since $C$ is contractible in the category of algebraic spaces,
	by Hartogs's lemma we have $H^{0}(X,\cO_{X})=H^{0}(X',\cO_{X'})$,
	where $X'=X\cup C=\ol{X}\setminus R$, and $\ol{\alpha}_{*}\cO_{X'}=\cO_{\mathbb{A}^1}$
	because $\ol{\alpha}|_{X'}\colon X'\to \mathbb{A}^1$ is proper and has connected fibres.
	Thus, $X$ is scheme-saturated, $\dim X^{\aff}=1$,
	but $X$ is not proper over $X^{\aff}$,
	because $\alpha^{-1}(0)=E\cap X\cong \mathbb{A}^1$.
\end{EXM}

\begin{EXM}[$\dim X^{\aff}=0$]
\label{ex:scheme-sat-0} 
	Let $\ol{X}\to \bP^2$ be a blowup at 10 points of a smooth cubic $\ol{C}$
	such that the proper transform of $C$ of $\ol{C}$ is not schematically contractible,
	see Lemma \ref{lem:Hironaka}\ref{item:non-contractible}.
	Let $X=\ol{X}\setminus C$. 
        Then $X$ is scheme-saturated, but not saturated.
	In fact, $X^{\sat}$ is proper, obtained from $\ol{X}$ by contracting $C$.
	Hence $X^{\aff}= (X^{\sat})^{\aff}$ is a point,
	but $X$ is not proper.
\end{EXM}

\bibliographystyle{amsalpha}
\bibliography{ref}

\end{document}